\documentclass[a4paper]{article}
\usepackage[margin=1in]{geometry}
\usepackage[utf8]{inputenc}
\usepackage{cite}
\usepackage{amsmath}
\usepackage{ntheorem}
\newtheorem{theorem}{Theorem}[section]
\newtheorem{lemma}{Lemma}[section]
\newtheorem*{proof}{Proof}[section]

\newtheorem{rem}{Remark}[section]
\newcommand{\RNum}[1]{\uppercase\expandafter{\romannumeral #1\relax}}
\usepackage{graphicx}
\usepackage{amssymb}
\usepackage[noend]{algpseudocode}
\usepackage{algorithmicx,algorithm}
\usepackage{xcolor}

\usepackage{bm}
\usepackage{appendix}
\usepackage{siunitx}
\usepackage{longtable,tabularx}
\usepackage{fontenc}
\usepackage{inputenc}
\usepackage{authblk}

\begin{document}

\title{Multi-level Monte Carlo path integral molecular dynamics  for thermal average calculation in the nonadiabatic regime}

\author[1]{Xiaoyu Lei}
\affil[1]{School of Mathematical Sciences, Peking University, Beijing 100871, China. \authorcr 1500010717@pku.edu.cn}

\author[2]{Zhennan Zhou}
\affil[2]{Beijing International Center for Mathematical Research, Peking University, Beijing 100871, P.R. China.\authorcr
zhennan@bicmr.pku.edu.cn}

\date{\today}

\maketitle

\begin{abstract}

With the path integral approach, the thermal average in multi-electronic-state quantum systems  can be approximated by the ring polymer representation on an extended configuration space, where the additional degrees of freedom are associated with the surface index of each bead.  The primary goal of this work is to propose a more efficient sampling algorithm for the calculation of such thermal averages. We reformulate the extended ring polymer approximation according to the configurations of the surface indexes, and by introducing a proper reference measure,  the reformulation is recast as a ratio of two expectations of function expansions. By quantitatively estimating the sub-estimators, and minimizing the total variance of the sampled average, we propose a multi-level Monte Carlo path integral molecular dynamics method (MLMC-PIMD) to achieve an optimal balance of computational cost and accuracy.

\end{abstract}

\section{Introduction}\label{sec-1}

%\zz{I have the impression that you misunderstand diabatic, adiabatic and nonadiabatic...}
%\zz{Let's zavoid using multi-level systems, since it is confusing since we have MLMC as well...}
\hspace{4mm}
Simulation of complex chemical system at the quantum level has always been a central and challenging task in the theoretical and computational chemistry, since a direct simulation of a quantum system is often numerically infeasible. While most numerical approaches are based on the Born-Oppenheimer approximation and the adiabatic assumption is usually taken, this assumption is no longer valid  when the interaction between multiple electronic energy surfaces cannot be neglected. In such scenarios, one needs to consider the multi-electronic-state systems. Readers can refer to \cite{makri1999time,stock2005classical,kapral2006progress} for more discussion.

\par  In this paper, we focus on the thermal average taking the following form
\begin{equation}\label{A}
    \langle\widehat{A}\rangle=\frac{\operatorname{Tr}_{n e}\left[e^{-\beta \widehat{H}} \widehat{A}\right]}{\operatorname{Tr}_{n e}\left[e^{-\beta \widehat{H}}\right]}
\end{equation}
where $\widehat{H}$ is a matrix-form Hamiltonian operator, $\widehat{A}$ is a matrix-form observable and $\beta$ is the inverse temperature given by $\frac{1}{k_{B}T}$ with $k_{B}$ the Boltzmann constant and T the temperature. And 
\[
\operatorname{Tr}_{n e}=\operatorname{Tr}_{n}\operatorname{Tr}_{e}=\operatorname{Tr}_{L^{2}\left(\mathbb{R}^{d}\right)}\operatorname{Tr}_{C^{2}}
\]
denotes the trace with respect to the nuclear and electronic degrees of freedom.  The prevailing numerical methods for thermal average calculation are based on the ring polymer representation, mapping the quantum particle to a ring polymer consists of its replica on the phase space\cite{berne1986simulation,markland2008efficient,lu2017path,LiuPath,lu2018accelerated,tao2018path}. The ring polymer representation was originally proposed in \cite{feynman1972lectures} and it then became the foundation for many numerical methods, which are mainly categorized  into two groups: the path integral Monte Carlo methods (see, e.g. \cite{chandler1981exploiting,berne1986simulation}) and the path integral molecular dynamics approaches (see, e.g. \cite{markland2008efficient,ceriotti2010efficient}). Besides, in recent years, in spite of vast applications in science, different mathematical aspects of the quantum thermal averages have been explored, such as the continuum limit, the preconditioning techniques \cite{lu2018continuum} and the Bayesian inversion problem \cite{chen2020bayesian}. %\zz{please add the refs}
\par
The conventional ring polymer representation does not directly apply to the nonadiabatic cases since multiple energy surfaces are involved.  One strategy to overcome this difficulty is to use the mapping variable approach\cite{stock1997semiclassical,meyera1979classical,stock2005classical}, and its basic idea is to replace the multi-electronic-state system by an augmented scalar system where the extra dimensions correspond to the electronic degrees of freedom \cite{stock1997semiclassical}. Another alternative strategy is to derive the extended ring polymer representation as in \cite{lu2017path,LiuPath}, following the spirit of the pioneering work of Schmidt and Tully \cite{schmidt2007path}, and the discreteness of the electronic states are preserved. In the extended ring polymer representation, each bead in the ring polymer is associated with a surface index showing which energy surface it lies in. And the sampling is then carried out in the extended space consisting of the position, momentum and surface index of each bead\cite{lu2017path}. Thus the thermal average is approximately transformed to the average over the extended configuration space of ring polymers. Another main contribution of \cite{lu2017path} is that a path integral molecular dynamics with surface hopping (abbreviated by PIMD-SH) dynamics sampling method is developed for sampling of the equilibrium distribution on the extended ring polymer space. The PIMD-SH dynamics ergodically samples the equilibrium distribution and it is shown to satisfy the detailed balance condition. When it comes to sampling off-diagonal observables, the straightforward PIMD-SH method becomes less efficient when sampling the configuration with kinks (a kink means the surface indexes of two consecutive beads are different), which is because when the configuration contains more kinks, it contributes exponentially less to the thermal average as shown in Section \ref{sec-3}. The infinite swapping limit of PIMD-SH was introduced and studied in \cite{lu2018accelerated} to remedy this issue by averaging over the surface indexes of the beads, of which the formulation essentially agrees with that in another independent work \cite{LiuPath}. In \cite{lu2018accelerated}, a multiscale integrator was proposed in the spirit of the heterogeneous multiscale method (abbreviated by HMM) \cite{e2011principles,vanden-eijnden2003,e2003,e33ddf9c3ffc44b78dd728bd6b4cff54} to improve the efficiency of the infinite swapping limit.

\par In this paper, we aim to further enhance the sampling efficiency of the computation of such thermal averages by leveraging the unique structure of the extended ring polymer representation of a multi-electronic-state system. It has been noticed in the previous work \cite{lu2017path} that when sampling the extended ring polymer representation, it is rare for a sampling path to visit the configurations with a large number of kinks, and the total contributions of such configurations are asymptotically small. This observation motivates us to consider the the extended ring polymer representation with a certain truncation. However, when the kinks are present, the contribution of off-diagonal component of the observable is amplified due to the kink energy. Thus, a rational and proper truncation is possible only one rewrites the extended ring polymer representation by the number of kinks and carry out detailed error analysis. 

Based on the reformation by the number of kinks, and the approximation estimates established in the error analysis, we are able to propose an improved sampling strategy in two steps. First, by introducing a proper reference measure, the reformulation can be viewed as a ratio of two expectations where the target functions to be sampled are in a form of series expansions with respect to the kink numbers respectively. Such a representation immediately implies a sampling method (to be specified in Section \ref{sec-4}), which we name the path integral molecular dynamics with a reference measure (abbreviated by RM-PIMD). Unlike the original PIMD-SH method, where the sampled value is asymptotically singular in the presence of kinks, in the RM-PIMD method, the proper introduction of the reference measure leads to a cancellation of the singular terms in the observable functions to be sampled, and thus, yields a more stable numerical performance. 

Next, by examining the sub-estimators in the expansions, we observe that as the kink number increases (while less than the half of the total bead number), the sampling difficulties increases dramatically while the variances of sub-estimators decreases asymptotically. Therefore, we adopt the spirit of the multi-level Monte Carlo method \cite{giles2013multilevel,Giles2008MultilevelMC,giles_2015}, and propose a second scheme, which we name MLMC-PIMD. It optimizes the numbers of samples allocated to each sub-estimator to minimize the variance of the total estimator, which is subject to the constraint that the total computational cost is fixed. In additional, the quantitative estimates in Section \ref{sec-3} guarantees that, a certain truncation at the number of kinks can lead to minimal error while easily enhancing the efficiency of both algorithms.

The paper is outlined as follows. In Section \ref{sec-2}, we give a brief review of the extended ring polymer representation for  two-state systems in the diabatic representation and the PIMD-SH method. We prove in Section \ref{sec-3} a quantitative error estimate for the truncated ring polymer approximation for thermal average. In Section \ref{sec-4}, based on the reference measure perspective, we further propose the Multi-level Monte Carlo path integral molecular dynamics (MLMC-PIMD) method  to minimize the variance with a given computational cost. In Section \ref{sec-5}, extensive numerical experiments are given to show the approximation property of truncated ring polymer representation and the validation of MLMC-PIMD method. In Section \ref{sec-6}, we summarize our new results and point out some possible directions for further research.

\section{Preliminary}\label{sec-2}
\subsection{Extended ring polymer representation for diabatic two-state systems}
\hspace{4mm}  In \eqref{A}, the Hamiltonian of a two-state system in a diabatic representation can be expressed as
\begin{equation*}
    \widehat{H}=\widehat{T}+\widehat{V}=\frac{1}{2 M}\left(\begin{array}{ll}{\hat{p}^{2}}& \\ &{\hat{p}^{2}}\end{array}\right)+\left(\begin{array}{ll}{V_{00}(\hat{q})} & {V_{01}(\hat{q})} \\ {V_{10}(\hat{q})} & {V_{11}(\hat{q})}\end{array}\right)
\end{equation*}
where $\hat{p}$ and $\hat{q}$ denote the momentum and position operators, and $M$ is the mass of nuclei (for simplicity, we assume all nuclei have the same mass). And the potential matrix
\begin{equation*}
    V(q)=\left(\begin{array}{ll}{V_{00}(q)} & {V_{01}(q)} \\ {V_{10}(q)} & {V_{11}(q)}\end{array}\right)
\end{equation*}
is a Hermitian matrix. For simplicity, we assume $V_{01}=V_{10}$, therefore they are real. In addition, we assume $V_{01}$ doesn't change sign for all $q$. Then the Hilbert space of the system is $L^{2}(\mathbb{R}^{d}) \otimes C^{2}$. For simplicity, we assume the matrix-form observable $\widehat{A}$ only depends on position $q$, which can be written as
\begin{equation*}
    \widehat{A}(\hat{q})=\left(\begin{array}{ll}{A_{00}(\hat{q})} & {A_{01}(\hat{q})} \\ {A_{10}(\hat{q})} & {A_{11}(\hat{q})}\end{array}\right).
\end{equation*}
\par
According to Section \RNum{2}A of \cite{lu2017path}, for a sufficiently large $N$, \eqref{A} can be approximated by an extended ring polymer as
\begin{gather}
     \operatorname{Tr}_{n e}\left[e^{-\beta \widehat{H}} \widehat{A}\right] \approx \frac{1}{(2 \pi)^{d N}}  \int_{\mathbb{R}^{2 d N}} \mathrm{d} \bm{q} \mathrm{d} \bm{p} \sum_{\bm{\ell} \in\{0,1\}^{N}}    e^{ -\beta_{N} H_{N}(\bm{q},\bm{p},\bm{\ell})} W_{N}[A](\bm{q},\bm{p},\bm{\ell})\label{num} \\
     \operatorname{Tr}_{n e}\left[e^{-\beta \widehat{H}}\right]  \approx \frac{1}{(2 \pi)^{d N}} Z_{N} :=\frac{1}{(2 \pi)^{d N}} \int_{\mathbb{R}^{2 d N}} \mathrm{d} \bm{q} \mathrm{d} \bm{p} \sum_{\bm{\ell} \in\{0,1\}^{N}}  e^{-\beta_{N} H_{N}(\bm{q},\bm{p},\bm{\ell})}\label{den}
\end{gather}
where $\beta_{N}=\frac{\beta}{N}$ and we assume $\hbar=1$ in the quantum system. Let $I$ denote the extended ring polymer representation with $N$ beads which we use to approximate the thermal average $\langle\widehat{A}\rangle$ we want to compute in this article. Namely, $I$ takes the form:
\begin{equation}\label{I}
I:=
\frac{\int_{\mathbb{R}^{2d N}} \mathrm{d} \bm{q} \mathrm{d} \bm{p} \sum_{\bm{\ell} \in\{0,1\}^{N}}    e^{ -\beta_{N} H_{N}(\bm{q},\bm{p},\bm{\ell})} W_{N}[A](\bm{q},\bm{p},\bm{\ell})}{\int_{\mathbb{R}^{2d N}} \mathrm{d} \bm{q} \mathrm{d} \bm{p} \sum_{\bm{\ell} \in\{0,1\}^{N}}  e^{-\beta_{N} H_{N}(\bm{q},\bm{p},\bm{\ell})}}
.
\end{equation}
Each bead is described by its position, momentum and surface index. The configuration of N-bead extended ring polymer representation $(\bm{q},\bm{p},\bm{\ell})=(q_{1},\dots,q_{N},p_{1},\dots,p_{N},\ell_{1},\dots,\ell_{N})$ ($  \bm{q}=(q_{1},\dots,q_{N}) \in R^{d N},\bm{p}=(p_{1},\dots,p_{N}) \in R^{d N}, \bm{\ell}=(\ell_{1},\dots,\ell_{N})\in \{0,1\}^{N} $) lies in the extended (ring polymer) configuration space $S:=\mathbb{R}^{2d N}\times\{0,1\}^{N}$, $N$ copies of phase space with surface indexes. And the Hamiltonian $H_{N}(\bm{q},\bm{p},\bm{\ell})$ is defined as
\begin{gather}
    H_{N}(\bm{q},\bm{p},\bm{\ell})=\sum_{k=1}^{N}\left\langle \ell_{k}\left|G_{k}\right| \ell_{k+1}\right\rangle\quad(\ell_{N+1}=\ell_{1}),\label{H}\\
    \left\langle \ell\left|G_{k}\right| \ell^{\prime}\right\rangle=\left\{\begin{array}{ll}{\frac{p_{k}^{2}}{2 M}+\frac{M\left(q_{k}-q_{k+1}\right)^{2}}{2\left(\beta_{N}\right)^{2}}+V_{\ell\ell}\left(q_{k}\right)-\frac{1}{\beta_{N}} \ln \left(\cosh \left(\beta_{N}\left|V_{01}\left(q_{k}\right)\right|\right)\right),} & {\ell=\ell^{\prime}} \\ {\frac{p_{k}^{2}}{2 M}+\frac{M\left(q_{k}-q_{k+1}\right)^{2}}{2\left(\beta_{N}\right)^{2}}+\frac{V_{00}\left(q_{k}\right)+V_{11}\left(q_{k}\right)}{2}-\frac{1}{\beta_{N}} \ln \left(\sinh \left(\beta_{N}\left|V_{01}\left(q_{k}\right)\right|\right)\right).} & {\ell \neq \ell^{\prime}}\end{array}\right.\label{h}
\end{gather}
\par  For observable $\widehat{A}(\hat{q})$, the function $W_{N}[A]$ takes the form:
\begin{equation}\label{W}
    W_{N}[A](\bm{q},\bm{p},\bm{\ell})=\frac{1}{N} \sum_{k=1}^{N}\left\langle \ell_{k}\left|A\left(q_{k}\right)\right| \ell_{k}\right\rangle -  e^{\beta_{N}\left\langle \ell_{k}\left|G_{k}\right| \ell_{k+1}\right\rangle-\beta_{N}\langle\overline{l}_{k}\left|G_{k}\right| \ell_{k+1}\rangle}\langle \ell_{k}\left|A\left(q_{k}\right)\right| \overline{\ell}_{k}\rangle \frac{V_{\ell_{k} \overline{\ell}_{k}}}
{\left|V_{\ell_{k} \overline{\ell}_{k}}\right|},
\end{equation}
where $\overline{\ell}_{k}=1-\ell_{k}$ is the surface index of the other potential energy surface and $\langle\ell|A|\ell^{\prime}\rangle$ is the corresponding element of the matrix-form observable $\widehat{A}$
\begin{equation*}
    \langle\ell|A(q)|\ell^{\prime}\rangle=A_{\ell\ell^{\prime}}(q)\quad\forall \ell,\ell^{\prime}\in \{0,1\}.
\end{equation*}
Different from the conventional ring polymer representation, each bead of extended ring polymer representation is associated with a surface index $\ell_{k}$ to show which energy surface it lies in. When $\ell_{k} \neq \ell_{k+1}$, we call it a kink in the extended ring polymer representation. It is easy to notice that when only two electronic states are involved the kink number is always an even number smaller than $N$ in a configuration. Readers can refer to Section \RNum{2}A of \cite{lu2017path} for more discussions about the extended ring polymer representation for the thermal average.

\subsection{A brief introduction to PIMD-SH method}
\hspace{4mm} 
To calculate the thermal average, from the extended ring polymer representation, one can reformulate the ratio of \eqref{num} to \eqref{den} as one expectation as in \cite{lu2017path,lu2018accelerated}. Then the extended ring polymer representation $I$ for thermal average $\langle\widehat{A}\rangle$ can be rewritten as
\begin{equation*}
    \langle\widehat{A}\rangle \approx I= \int_{\mathbb{R}^{2d N}} \mathrm{d} \bm{q} \mathrm{d} \bm{p} \sum_{\bm{\ell} \in\{0,1\}^{N}} \pi(\tilde{\bm{z}}) W_{N}[A](\tilde{\bm{z}}),
\end{equation*}
where $\bm{z}$ and $\tilde{\bm{z}}$ respectively denote $(\bm{q},\bm{p})$ in the position and momentum space $\mathbb{R}^{2d N}$ and $(\bm{q},\bm{p},\bm{\ell})$ in the extended ring polymer configuration space $S=\mathbb{R}^{2d N}\times \{0,1\}^{N}$. And $\pi (\tilde{\bm{z}})$ is a distribution on extended configuration space $S$ taking the form as 
\begin{equation*}
    \pi(\tilde{\bm{z}})=\frac{1}{Z_{N}} e^{-\beta_{N} H_{N}(\tilde{\bm{z}})}.
\end{equation*}
\par  The PIMD-SH method proposed a sampling scheme $\tilde{\bm{z}}(t)$, a stochastic differential whose trajectory is ergodic with respect to equilibrium distribution $\pi$, then the integral of $W_{N}[A]$ with respect to the distribution $\pi$ can be approximated by sampling according to the trajectory of $\tilde{\bm{z}}(t)$:
\begin{equation*}
    \int_{\mathbb{R}^{2d N}} \mathrm{d} \bm{q} \mathrm{d} \bm{p} \sum_{\bm{\ell} \in\{0,1\}^{N}} \pi(\tilde{\bm{z}}) W_{N}[A](\tilde{\bm{z}})\approx
    \lim _{T \rightarrow \infty} \frac{1}{T} \int_{0}^{T} W_{N}[A](\tilde{\bm{z}}(t)) \mathrm{d} t.
\end{equation*}
The trajectory $\tilde{\bm{z}}(t)$ is constructed as following:
\begin{equation}\label{PIMD-SH}
    \left\{\begin{array}{l}{\mathrm{d} \bm{q}=\nabla_{\bm{p}} H_{N}(\bm{q}(t), \bm{p}(t), \bm{\ell}(t)) \mathrm{d} t}, \\ {\mathrm{d} \bm{p}=-\nabla_{\bm{q}} H_{N}(\bm{q}(t), \bm{p}(t), \bm{\ell}(t)) \mathrm{d} t-\gamma \bm{p} \mathrm{d} t+\sqrt{2 \gamma \beta_{N}^{-1} M} \mathrm{d} \bm{B}}, \\ {P\left(\bm{\ell}(t+\delta t)=\bm{\ell}^{\prime} | \bm{\ell}(t)=\bm{\ell}, \bm{z}(t)=\bm{z}\right)=\delta_{\bm{\ell}^{\prime}, \bm{\ell}}+\eta \lambda_{\bm{\ell}^{\prime}, \bm{\ell}}(\bm{z}) \delta t+o(\delta t),}\end{array}\right.
\end{equation}
where $\mathrm{d} \bm{B}$ is Brownian motion of dimension $d N$, $\gamma \geq 0$ denotes the friction constant, $\eta >0$ serves as an overall scaling parameter for the hopping intensity and $\delta t \ll 1$ denotes the infinitesimal time interval. Notice the last line of \eqref{PIMD-SH} is established in the sense of the limitation $\delta t\rightarrow 0$. The coefficients $\lambda_{\bm{\ell}^{\prime},\bm{\ell}}$ are defined as
\begin{equation*}
    \lambda_{\bm{\ell}^{\prime}, \bm{\ell}}=\left\{\begin{array}{ll}{-\sum_{\tilde{\bm{\ell}} \in S_{\bm{\ell}}} p_{\tilde{\bm{\ell}}, \bm{\ell}}(\bm{z})} & {\bm{\ell}^{\prime}=\bm{\ell}} \\ {p_{\bm{\ell}^{\prime}, \bm{\ell}}} & {\bm{\ell}^{\prime} \in S_{\bm{\ell}}} \\ {0} & {\text {otherwise }}\end{array}\right.
\end{equation*}
where
\begin{equation*}
   S_{\bm{\ell}}=\left\{\bm{\ell}^{\prime} |\left\|\bm{\ell}^{\prime}-\bm{\ell}\right\|_{1}=1  \text { or } \bm{\ell}^{\prime}=\bm{1}-\bm{\ell}\right\} \,( ||\bm{\ell}^{\prime}-\bm{\ell}||_{1}=\sum_{k=1}^{N}|\ell^{\prime}_{k}-\ell_{k}|)\quad\text{and}\quad
    p_{\bm{\ell}^{\prime}, \bm{\ell}}=e^{\frac{\beta_{N}}{2}\left(H_{N}(\bm{z}, \bm{\ell})-H_{N}\left(\bm{z}, \bm{\ell}^{\prime}\right)\right)} .
\end{equation*}
The evolution of $\bm{\ell}(t)$ is a Markov jump process following a surface hopping type dynamics. Readers can refer to the work of the fewest switches surface hopping in \cite{tully1990molecular} (and other recent works \cite{lu2016improved,lu2018frozen}). And the choice of $p_{\bm{\ell}^{\prime},\bm{\ell}}$ satisfies the detailed balance condition in order to preserve the distribution $\pi$ under the dynamics. The choice of $S_{\bm{\ell}}$ only allows two types of changes in the surface index sequence $\bm{\ell}$: first, only one bead flips to the other energy surface; second, all beads in the sequence flip to contrary energy surface. Although the choice of $S_{\bm{\ell}}$ is for simplicity, it can guarantee that the jump process $\bm{\ell}(t)$ can reach any surface index configuration.
It has been proved in Section \RNum{2}B of \cite{lu2017path} that the trajectory $\tilde{\bm{z}}(t)$ has the ergodic property to sample the distribution $\pi$ on the extended configuration space $S$. Because it shows the distribution $\pi$ is a stationary solution to the Fokker-Planck equation of the process $\tilde{\bm{z}}(t)$. Readers can refer to Section \RNum{2}B of \cite{lu2017path} for more discussion about PIMD-SH method.

\section{The truncated thermal averages}\label{sec-3}
\hspace{4mm} 

%\zz{need to introduce simplified notations and make the theorems readable: highlight key definitions and add necessary interpretations}

%\zz{fix the thoerum number references, use Theorem, not theorem}

Our motivation to improve the sampling of the thermal average calculation comes from leveraging the unique structure of the extended ring polymer representation. To demonstrate our insight in a heuristic way, we introduce some notations below. Let 
$$T(\bm{\ell})=\int_{\mathbb{R}^{2dN}}e^{-\beta_{N}H_{N}(\bm{q},\bm{p},\bm{\ell})}\mathrm{d}\bm{q}\mathrm{d}\bm{p} \quad\text{and}\quad
T_{A}(\bm{\ell})= \int_{\mathbb{R}^{2dN}} W_{N}[A](\bm{q},\bm{p},\bm{\ell})e^{-\beta_{N}H_{N}(\bm{q},\bm{p},\bm{\ell})}\mathrm{d}\bm{q}\mathrm{d}\bm{p}.
$$
With some simple calculation, $I$ as in \eqref{I} can be rewritten as
\begin{equation}\label{I_2}
I= \frac{\sum\limits_{k=0}^{\lfloor\frac{N}{2}\rfloor}\sum\limits_{|\bm{\ell}|=2k}T_{A}(\bm{\ell})}{\sum\limits_{k=0}^{\lfloor\frac{N}{2}\rfloor}\sum\limits_{|\bm{\ell}|=2k}T(\bm{\ell})},
\end{equation}
where $|\bm{\ell}|$ denotes the kink number of surface index sequence $\bm{\ell}\in \{0,1\}^{N}$.

\par The distribution $\pi$ can be rewritten as
\begin{equation}\label{pi}
    \pi(\tilde{\bm{z}})=\frac{1}{Z_{N}}e^{-\frac{\beta_{N}}{2 M} \sum_{k=1}^{N} p_{k}^{2}-\frac{M}{2 \beta_{N}} \sum_{k=1}^{N}\left(q_{k}-q_{k+1}\right)^{2}-\beta_{N} \sum_{k=1}^{N} V\left(q_{k}, \ell_{k}, \ell_{k+1}\right)} \prod\limits_{k=1}^{N} F\left(q_{k}, \ell_{k}, \ell_{k+1}\right)
\end{equation}
according to \eqref{H} and \eqref{h}, where
\begin{equation*}
    V\left(q_{k}, \ell_{k}, \ell_{k+1}\right)=
    \begin{cases}
    V_{\ell_{k}\ell_{k}}(q_{k})&\ell_{k}=\ell_{k+1}\\
    \frac{V_{00}(q_{k})+V_{11}(q_{k})}{2}&\ell_{k}\neq \ell_{k+1}
    \end{cases} 
    \, \text{and} \,
    F(q_{k},\ell_{k},\ell_{k+1})=
    \begin{cases}
    \cosh(\beta_{N}|V_{01}(q_{k})|)&\ell_{k}=\ell_{k+1}\\
    \sinh(\beta_{N}|V_{01}(q_{k})|)&\ell_{k}\neq \ell_{k+1}
    \end{cases}.
\end{equation*}
According to the special form of $F(q,\ell,\ell^{\prime})$, for large $N$ and small $|V_{01}(q)|$, we have
\begin{equation*}
    F(q,\ell,\ell^{\prime})=
    \begin{cases}
    \cosh(\beta_{N}|V_{01}(q)|) \approx 1 & \ell=\ell^{\prime}\\
    \sinh(\beta_{N}|V_{01}(q)|) \approx \beta_{N}|V_{01}(q)|\approx \frac{C}{N} & \ell\neq \ell^{\prime}
    \end{cases}.
\end{equation*}
Thus while the kink number of $\tilde{\bm{z}}$ increases, the value $\pi(\tilde{\bm{z}})$ decreases exponentially, which means the integral $T(\bm{\ell})$ is negligible when $\bm{\ell}$ has a large number of kinks.
With the same idea, the integral $T_{A}(\bm{\ell})$ can be exponentially small when $\bm{\ell}$ contains many kinks and $W_{N}[A]$ is bounded. Detailed analysis will be shown in the proof of Theorem \ref{thm-1}. Enlightened by this observation, the extended ring polymer representation $I$ can be approximated by the following
\begin{equation}\label{I_tru}
    I\approx I_{2k_{0}}:=\frac{\sum\limits_{k=0}^{k_{0}}\sum\limits_{|\bm{\ell}|=2k}T_{A}(\bm{\ell})}{\sum\limits_{k=0}^{k_{0}}\sum\limits_{|\bm{\ell}|=2k}T(\bm{\ell})},
\end{equation}
and we name $I_{2k_{0}}$ the truncated thermal average.

\par We shall show in Theorem \ref{thm-1}  the truncated thermal average $I_{2k_{0}}$ can approximate the extended ring polymer representation $I$ when $k_{0}$ is properly chosen. First, we have the following lemma counting the number of configurations given the number of kinks.
\begin{lemma}\label{lem-1}
For integer $N$ and $k\,(0\leq k\leq \lfloor\frac{N}{2}\rfloor)$, $\{0,1\}^N$ contains $2\tbinom{N}{2k}$ different surface index sequences which have $2k$ kinks. 
\end{lemma}
\begin{proof}
When the beads number is $N$, we can determine the surface index sequence uniquely after we know the first number is $0$ or $1$ and where the kinks happen. Since there are $N$ intervals for kinks to happen (the kinks occur between $\ell_{k}$ and $\ell_{k+1}$ ($1\leq k\leq N$)), the total number of $2k$-kink sequences are $2\tbinom{N}{2k}$. We complete the proof. $\hfill\square$
\end{proof}

\par Thus we are ready to state the main result of this section.
\begin{theorem}\label{thm-1}
Consider the thermal average in the ring polymer representation as in  \eqref{I}. Suppose the
the diagonal potentials satisfy: $V_{00}(q)-V_{11}(q)\le C_{1}$, the off-diagonal potential is bounded: $0<V_{01}(q)<C_{2}$ and the observable $\widehat{A}$ is bounded in each element: $|\langle \ell|A(q)| \ell^{\prime}\rangle| <C_{3},\,\forall \ell,\ell^{\prime}\in \{0,1\}$, where $C_{1}$, $C_{2}$ and $C_{3}$ are some generic constants. Then we have 
\begin{equation*}
\left|I-I_{2 k_{0}}\right| \leq C_{4} C_{5}^{N}N\sum\limits
_{k=k_{0}+1}^{\left\lfloor\frac{N}{2}\right\rfloor} \frac{1}{(2 k) !}
\end{equation*}
for some constants $C_{4}$ and $C_{5}$ both independent of $N$ and $k_{0}$. 
\end{theorem}
\begin{proof}
To simplify our calculation below, let $\beta=1$ or equivalently, $\beta_{N}=\frac{1}{N}$. Let $C_{1}>0$ and $C_{2}>1$, because if these conditions are not satisfied, we can replace $C_{1}$ and $C_{2}$ respectively by $|C_{1}|$ and $C_{2}\vee 1$ and the proof below also works. Define
	\begin{equation*}
	T_{A}(\bm{\ell})=\int_{\mathbb{R}^{2dN}}W_{N}[A](\bm{q},\bm{p},\bm{\ell})e^{-\beta_{N}H_{N}(\bm{q},\bm{p},\bm{\ell})}\mathrm{d}\bm{q}\mathrm{d}\bm{p}
	\quad\text{and}\quad
	T(\bm{\ell})=\int_{\mathbb{R}^{2dN}}e^{-\beta_{N}H_{N}(\bm{q},\bm{p},\bm{\ell})}\mathrm{d}\bm{q}\mathrm{d}\bm{p}.
	\end{equation*}
	According to the definition of $W_{N}[A]$ in \eqref{W}, 
	\begin{align*}
	T_{A}(\bm{\ell})&=\int_{\mathbb{R}^{2dN}}W_{N}[A](\bm{q},\bm{p},\bm{\ell})e^{-\beta_{N}H_{N}(\bm{q},\bm{p},\bm{\ell})}\mathrm{d}\bm{q}\mathrm{d}\bm{p}\\
	&=\int_{\mathbb{R}^{2dN}}\left(\frac{1}{N} \sum_{k=1}^{N}\left\langle \ell_{k}\left|A\left(q_{k}\right)\right| \ell_{k}\right\rangle\right)e^{-\beta_{N}H_{N}(\bm{q},\bm{p},\bm{\ell})}\mathrm{d}\bm{q}\mathrm{d}\bm{p}-\\
	&-\frac{1}{N}\sum\limits_{k=1}^{k=N}\int_{\mathbb{R}^{2dN}}e^{\beta_{N}\left\langle \ell_{k}\left|G_{k}\right| \ell_{k+1}\right\rangle-\beta_{N}\left\langle\bar{\ell}_{k}\left|G_{k}\right| \ell_{k+1}\right\rangle}\left\langle \ell_{k}\left|A\left(q_{k}\right)\right| \bar{\ell}_{k}\right\rangle e^{-\beta_{N}H_{N}(\bm{q},\bm{p},\bm{\ell})}\mathrm{d}\bm{q}\mathrm{d}\bm{p}\\
	&=\int_{\mathbb{R}^{2dN}}\left(\frac{1}{N} \sum_{k=1}^{N}\left\langle \ell_{k}\left|A\left(q_{k}\right)\right| \ell_{k}\right\rangle\right)e^{-\beta_{N}H_{N}(\bm{q},\bm{p},\bm{\ell})}\mathrm{d}\bm{q}\mathrm{d}\bm{p}-\\
	&-\frac{1}{N}\sum\limits_{k=1}^{N}\int_{\mathbb{R}^{2dN}}\left\langle \ell_{k}\left|A\left(q_{k}\right)\right| \bar{\ell}_{k}\right\rangle
	e^{-\beta_{N}\left\langle\bar{\ell}_{k}\left|G_{k}\right| \ell_{k+1}\right\rangle-\beta_{N}\sum\limits_{1\le j\le N,j\neq k}\left\langle\bar{\ell}_{j}\left|G_{k}\right| \ell_{j+1}\right\rangle}\mathrm{d}\bm{q}\mathrm{d}\bm{p}.
	\end{align*}
	Define
	\begin{equation*}
	T_{B}(\bm{\ell})=\int_{\mathbb{R}^{2dN}}\left(\frac{1}{N} \sum_{k=1}^{N}\left\langle \ell_{k}\left|A\left(q_{k}\right)\right| \ell_{k}\right\rangle\right)e^{-\beta_{N}H_{N}(\bm{q},\bm{p},\bm{\ell})}\mathrm{d}\bm{q}\mathrm{d}\bm{p},
	\end{equation*}
	\begin{equation*}
	T_{C,k}(\bm{\ell})=\int_{\mathbb{R}^{2dN}}\left\langle \ell_{k}\left|A\left(q_{k}\right)\right| \bar{\ell}_{k}\right\rangle
	e^{-\beta_{N}\left\langle\bar{\ell}_{k}\left|G_{k}\right| \ell_{k+1}\right\rangle-\beta_{N}\sum\limits_{1\le j\le N,j\neq k}\left\langle\bar{\ell}_{j}\left|G_{k}\right| \ell_{j+1}\right\rangle}\mathrm{d}\bm{q}\mathrm{d}\bm{p}\quad\text{and}
	\end{equation*}
	\begin{equation*}
	T_{k}(\bm{\ell})=\int_{\mathbb{R}^{2dN}}
	e^{-\beta_{N}\left\langle\bar{\ell}_{k}\left|G_{k}\right| \ell_{k+1}\right\rangle-\beta_{N}\sum\limits_{1\le j\le N,j\neq k}\left\langle\bar{\ell}_{j}\left|G_{k}\right| \ell_{j+1}\right\rangle}\mathrm{d}\bm{q}\mathrm{d}\bm{p}.
	\end{equation*}
	We can write $T_{A}(\bm{\ell})$ to
	\begin{align}\label{p-1}
	T_{A}(\bm{\ell})=T_{B}(\bm{\ell})-\frac{1}{N}\sum\limits_{k=1}^{N}T_{C,k}(\bm{\ell}).
	\end{align}
	Notice $|\langle \ell|A(q)| \ell^{\prime}\rangle| <C_{3}\,(\forall \ell,\ell^{\prime}\in \{0,1\})$, thus
	\begin{equation}\label{p-2}
	\left|\frac{1}{N} \sum_{k=1}^{N}\left\langle \ell_{k}\left|A\left(q_{k}\right)\right| \ell_{k}\right\rangle\right| \le \frac{1}{N} \sum_{k=1}^{N}\left|\left\langle \ell_{k}\left|A\left(q_{k}\right)\right| \ell_{k}\right\rangle\right| \le C_{3}
    \quad\text{and}\quad
	|T_{B}(\bm{\ell})|\le C_{3}T(\bm{\ell}).
	\end{equation}
	By the same calculation, we have
	\begin{equation}\label{p-3}
	|T_{C,k}(\bm{\ell})|\le \int_{\mathbb{R}^{2dN}}\left|\left\langle \ell_{k}\left|A\left(q_{k}\right)\right| \bar{\ell}_{k}\right\rangle\right|
	e^{-\beta_{N}\left\langle\bar{\ell}_{k}\left|G_{k}\right| \ell_{k+1}\right\rangle-\beta_{N}\sum\limits_{1\le j\le N,j\neq k}\left\langle\bar{\ell}_{j}\left|G_{k}\right| \ell_{j+1}\right\rangle}\mathrm{d}\bm{q}\mathrm{d}\bm{p}
	\le C_{3}T_{k}(\bm{\ell}).
	\end{equation}
	We assume $|\bm{\ell}|=2k_{0}$ and let the $2k_{0}$ kinks of $\bm{\ell}$ happen after surface indexes $\ell_{i_{1}}, \ldots, \ell_{i_{2 k_{0}}}\,(1 \leq i_{1}<\cdots<i_{2k_{0}}\le N)$. Let $M=\cosh(C_{2}\beta_{N})$, we have
	\begin{equation*}
	\frac{\sinh(\beta_{N}|V_{01}(q_{k})|)}{\beta_{N}|V_{01}(q_{k})|}=
	\cosh (\theta\beta_{N}|V_{01}(q_{k})|)\leq \cosh (C_{2}\beta_{N})=M,\quad (0\leq \theta\leq 1)
	\end{equation*}
	and thus 
	\begin{equation}\label{p-4}
	\sinh(\beta_{N}|V_{01}(q_{k})|)\le M\beta_{N}|V_{01}(q_{k})|\le MC_{2}\beta_{N}.
	\end{equation}
	According to \eqref{p-4}, we have
	\begin{align}\label{p-5}
	T(\bm{\ell})&= \int_{R^{2d N}}e^{-\frac{\beta_{N}}{2M}\sum\limits_{k=1}^{N}p_{k}^2-\frac{M}{2\beta_{N}}\sum\limits_{k=1}^{N}(q_{k}-q_{k+1})^{2}-\beta_{N}\sum\limits_{k=1}^{N}V(q_{k},\ell_{k},\ell_{k+1})}\prod\limits_{k=1}^{2k_{0}}\sinh(\beta_{N}|V_{01}(q_{i_{k}})|)\times \notag \\
	&\times\prod\limits_{k\neq i_{1},\cdots,i_{2k_{0}}}\cosh(\beta_{N}|V_{01}(q_{k})|)\mathrm{d}\bm{q}\mathrm{d}\bm{p}
	\notag \\
	&\leq \int_{R^{2d N}}e^{-\frac{\beta_{N}}{2M}\sum\limits_{k=1}^{N}p_{k}^2-\frac{M}{2\beta_{N}}\sum\limits_{k=1}^{N}(q_{k}-q_{k+1})^{2}-\beta_{N}\sum\limits_{k=1}^{N}V(q_{k},\ell_{k},\ell_{k+1})}\prod\limits_{k=1}^{2k_{0}}M\beta_{N}|V_{01}(q_{i_{k}})|\times M^{N-2k_{0}}\mathrm{d}\bm{q}\mathrm{d}\bm{p} \notag \\
	&\leq M^N \beta_{N}^{2k_{0}}C_{2}^{2k_{0}}C(\bm{\ell})\le M^{N}C_{2}^N\beta_{N}^{2k_{0}}C(\bm{\ell}),
	\end{align}
	where $C(\bm{\ell})=\int_{R^{2d N}}e^{-\frac{\beta_{N}}{2M}\sum\limits_{k=1}^{N}p_{k}^2-\frac{M}{2\beta_{N}}\sum\limits_{k=1}^{N}(q_{k}-q_{k+1})^{2}-\beta_{N}\sum\limits_{k=1}^{N}V(q_{k},\ell_{k},\ell_{k+1})}\mathrm{d}\bm{q}\mathrm{d}\bm{p}.$\\
    For $T_{n}(\bm{\ell})$, we also have
	\begin{align}
	T_{n}(\bm{\ell})&=\int_{R^{2d N}}e^{-\frac{\beta_{N}}{2M}\sum\limits_{k=1}^{N}p_{k}^2-\frac{M}{2\beta_{N}}\sum\limits_{k=1}^{N}(q_{k}-q_{k+1})^{2}-\beta_{N}V(q_{n},\bar{\ell}_{n},\ell_{n+1})-\beta_{N}\sum\limits_{k\neq n,1\le k\le N}V(q_{k},\ell_{k},\ell_{k+1})}\times \notag \\
	&\times F(q_{n},\bar{\ell}_{n},\ell_{n+1}) \prod\limits_{k\neq n,1\le k\le N} F\left(q_{k}, \ell_{k}, \ell_{k+1}\right)
	\mathrm{d}\bm{q}\mathrm{d}\bm{p}.
	\notag
	\end{align}
    We observe that when $\bm{\ell}$ contains $2k_{0}$ kinks, $\prod\limits_{k=1}^{N} F\left(q_{k}, \ell_{k}, \ell_{k+1}\right)$ contains $2k_{0}$ $\sinh$ terms, each of which is smaller than $MC_{2}\beta_{N}$, and $(N-2k_{0})$ $\cosh$ terms, each of which is samller than $M$. Compared with $\prod\limits_{k=1}^{N} F\left(q_{k}, \ell_{k}, \ell_{k+1}\right)$, $F(q_{n},\bar{\ell}_{n},\ell_{n+1}) \prod\limits_{k\neq n,1\le k\le N} F\left(q_{k}, \ell_{k}, \ell_{k+1}\right)$ contains $(2k_{0}+\Delta)$ $\sinh$ terms and $(N-2k_{0}-\Delta)$ $\cosh$ terms, where $\Delta\in \{-1,1\}$ dependent on the choice of $\ell_{n}$ and $\ell_{n+1}$, as a result,
	\begin{align}\label{p-6}
	T_{n}(\bm{\ell})&=\int_{R^{2d N}}e^{-\frac{\beta_{N}}{2M}\sum\limits_{k=1}^{N}p_{k}^2-\frac{M}{2\beta_{N}}\sum\limits_{k=1}^{N}(q_{k}-q_{k+1})^{2}-\beta_{N}V(q_{n},\bar{\ell}_{n},\ell_{n+1})-\beta_{N}\sum\limits_{k\neq n,1\le k\le N}V(q_{k},\ell_{k},\ell_{k+1})}\times \notag \\
	&\times F(q_{n},\bar{\ell}_{n},\ell_{n+1}) \prod\limits_{k\neq n,1\le k\le N} F\left(q_{k}, \ell_{k}, \ell_{k+1}\right)
	\mathrm{d}\bm{q}\mathrm{d}\bm{p} \notag\\
	&\le \int_{R^{2d N}}e^{-\frac{\beta_{N}}{2M}\sum\limits_{k=1}^{N}p_{k}^2-\frac{M}{2\beta_{N}}\sum\limits_{k=1}^{N}(q_{k}-q_{k+1})^{2}-\beta_{N}V(q_{n},\bar{\ell}_{n},\ell_{n+1})-\beta_{N}\sum\limits_{k\neq n,1\le k\le N}V(q_{k},\ell_{k},\ell_{k+1})}\mathrm{d}\bm{q}\mathrm{d}\bm{p}\times\notag\\
	&\times (MC_{2}\beta_{N})^{2k_{0}+\Delta}M^{N-2k_{0}-\Delta}\notag\\
	&= M^{N}C_{2}^{2k_{0}+\Delta}\beta_{N}^{2k_{0}+\Delta}C_{n}(\bm{\ell})\le M^{N}C_{2}^{N}N\beta_{N}^{2k_{0}}C_{n}(\bm{\ell}),
	\end{align}
	where $C_{n}(\bm{\ell})=\int_{R^{2d N}}e^{-\frac{\beta_{N}}{2M}\sum\limits_{k=1}^{N}p_{k}^2-\frac{M}{2\beta_{N}}\sum\limits_{k=1}^{N}(q_{k}-q_{k+1})^{2}-\beta_{N}V(q_{n},\bar{\ell}_{n},\ell_{n+1})-\beta_{N}\sum\limits_{k\neq n,1\le k\le N}V(q_{k},\ell_{k},\ell_{k+1})}\mathrm{d}\bm{q}\mathrm{d}\bm{p}.$\\
	Let $\bm{\ell}_{0}=\{0,\cdots,0\}$, another useful observation is
	\begin{equation}\label{p-7}
	T(\bm{\ell}_{0})= \int_{R^{2d N}}e^{-\frac{\beta_{N}}{2M}\sum\limits_{k=1}^{N}p_{k}^2-\frac{M}{2\beta_{N}}\sum\limits_{k=1}^{N}(q_{k}-q_{k+1})^{2}-\beta_{N}\sum\limits_{k=1}^{N}V(q_{k},0,0)}
	\prod\limits_{k=1}^{N}\cosh(\beta_{N}|V_{01}(q_{k})|)\mathrm{d}\bm{q}\mathrm{d}\bm{p}
	\ge C(\bm{\ell}_{0}),
	\end{equation}because $\cosh(\beta_{N}|V_{01}(q_{k})|)\ge 1$.
	When $\bm{\ell}$ contains $2k$ kinks, according to \eqref{p-2} and \eqref{p-5}, we have
	\begin{equation}\label{p-8}
	\left|T_{B}(\bm{\ell})\right| \leq C_{3} T(\bm{\ell}) \le C_{3}M^{N}C_{2}^{N} \beta_{N}^{2 k} C(\bm{\ell}).
	\end{equation}
	According to \eqref{p-3} and \eqref{p-6}, we have
	\begin{equation}\label{p-9}
	\left|T_{C, n}(\bm{\ell})\right|\le C_{3} T_{n}(\bm{\ell}) \le C_{3}M^{N}C_{2}^{N} N \beta_{N}^{2 k} C_{n}(\bm{\ell}).
	\end{equation}
	Under the assumption $V_{00}(q)-V_{11}(q)<C_{1}$, we can control $C(\bm{\ell})$ and $C_{n}(\bm{\ell})$ using $C(\bm{\ell}_{0})$ because
	\begin{align}
	\frac{C(\bm{\ell})}{C(\bm{\ell}_{0})}&= \frac{\int_{R^{d N}}e^{-\frac{M}{2\beta_{N}}\sum\limits_{k=1}^{N}(q_{k}-q_{k+1})^{2}-\beta_{N}\sum\limits_{k=1}^{N}V(q_{k},\ell_{k},\ell_{k+1})}\mathrm{d}\bm{q}}{\int_{R^{d N}}e^{-\frac{M}{2\beta_{N}}\sum\limits_{k=1}^{N}(q_{k}-q_{k+1})^{2}-\beta_{N}\sum\limits_{k=1}^{N}V_{00}(q_{k})}\mathrm{d}\bm{q}} \notag
	\\
	&=\frac{\int_{R^{d N}}e^{-\frac{M}{2\beta_{N}}\sum\limits_{k=1}^{N}(q_{k}-q_{k+1})^{2}-\beta_{N}\sum\limits_{k=1}^{N}V_{00}(q_{k})}e^{\beta_{N}\sum\limits_{k=1}^{N}\left(V_{00}(q_{k})-V(q_{k},\ell_{k},\ell_{k+1})\right)}\mathrm{d}\bm{q}}{\int_{R^{d N}}e^{-\frac{M}{2\beta_{N}}\sum\limits_{k=1}^{N}(q_{k}-q_{k+1})^{2}-\beta_{N}\sum\limits_{k=1}^{N}V_{00}(q_{k})}\mathrm{d}\bm{q}}, \notag
	\end{align}
	and
	\begin{equation}\label{p-10}
	V_{00}\left(q_{k}\right)-V\left(q_{k}, \ell_{k}, \ell_{k+1}\right)=\left\{\begin{array}{ll}
	0 \leq C_{1}, & \ell_{k}=\ell_{k+1}=0 \\
	V_{00}\left(q_{k}\right)-V_{11}\left(q_{k}\right) \leq C_{1}, & \ell_{k}=\ell_{k+1}=1 \\
	\frac{1}{2}\left(V_{00}\left(q_{k}\right)-V_{11}\left(q_{k}\right)\right) \leq C_{1}. & \ell_{k} \neq \ell_{k+1}
	\end{array}\right.
	\end{equation}
	Notice from \eqref{p-10}, $V_{00}\left(q_{k}\right)-V\left(q_{k}, \ell_{k}, \ell_{k+1}\right)$ can be bounded from above only if we assume $V_{00}(q)-V_{11}(q)$ is bounded from above because the right side of \eqref{p-10} only contains $V_{00}(q)-V_{11}(q)$ but not the opposite direction $V_{11}(q)-V_{00}(q)$. Thus $\beta_{N}\sum\limits_{k=1}^{N}\left(V_{00}(q_{k})-V(q_{k},\ell_{k},\ell_{k+1})\right)\le \beta_{N}\sum\limits_{k=1}^{N}C_{1}=\beta_{N}N C_{1}=C_{1}$ noticing $\beta_{N}=\frac{1}{N}$, from which we have $e^{\beta_{N}\sum\limits_{k=1}^{N}\left(V_{00}(q_{k})-V(q_{k},\ell_{k},\ell_{k+1})\right)}\le 
	e^{\beta_{N}\sum\limits_{k=1}^{N}C_{1}}=e^{C_{1}}$, and 
	\begin{equation}\label{p-11}
	\frac{C(\bm{\ell})}{C(\bm{\ell}_{0})}\le e^{C_{1}}.
	\end{equation}
	By the same analysis,
	\begin{align}
	\frac{C_{n}(\bm{\ell})}{C(\bm{\ell}_{0})}
	&=\frac{\int_{R^{d N}}e^{-\frac{M}{2\beta_{N}}\sum\limits_{k=1}^{N}(q_{k}-q_{k+1})^{2}-\beta_{N}V(q_{n},\bar{\ell}_{n},\ell_{n+1})-\beta_{N}\sum\limits_{k\neq n,1\le k\le N}V(q_{k},\ell_{k},\ell_{k+1})}\mathrm{d}\bm{q}}{\int_{R^{d N}}e^{-\frac{M}{2\beta_{N}}\sum\limits_{k=1}^{N}(q_{k}-q_{k+1})^{2}-\beta_{N}\sum\limits_{k=1}^{N}V_{00}(q_{k})}\mathrm{d}\bm{q}}\notag \\
	&=\frac{\int_{R^{d N}}e^{-\frac{M}{2\beta_{N}}\sum\limits_{k=1}^{N}(q_{k}-q_{k+1})^{2}-\beta_{N}\sum\limits_{k=1}^{N}V_{00}(q_{k})+\beta_{N}(V_{00}(q_{n})-V(q_{n},\bar{\ell}_{n},\ell_{n+1}))+\beta_{N}\sum\limits_{k\neq n,1\le k\le N}(V_{00}(q_{k})-V(q_{k},\ell_{k},\ell_{k+1}))}\mathrm{d}\bm{q}}{\int_{R^{d N}}e^{-\frac{M}{2\beta_{N}}\sum\limits_{k=1}^{N}(q_{k}-q_{k+1})^{2}-\beta_{N}\sum\limits_{k=1}^{N}V_{00}(q_{k})}\mathrm{d}\bm{q}}.\notag
	\end{align}
	According to \eqref{p-10}, we also have
	\begin{equation}\label{p-12}
	\frac{C_{n}(\bm{\ell})}{C(\bm{\ell}_{0})}\le e^{C_{1}}.
	\end{equation}
	With those estimates we can begin to bound $\left|I-I_{2 k_{0}}\right|$. According to the definition of $I$ and $I_{2k_{0}}$, we have 
	\begin{equation}\label{p-13}
	I=\frac{\sum\limits_{k=0}^{\left\lfloor\frac{N}{2}\right\rfloor} \sum\limits_{|\bm{\ell}|=2 k} T_{A}(\bm{\ell})}{\sum\limits_{k=0}^{\left\lfloor\frac{N}{2}\right\rfloor} \sum\limits_{|\bm{\ell}|=2 k} T(\bm{\ell})} \quad \text { and } \quad I_{2 k_{0}}=\frac{\sum\limits_{k=0}^{k_{0}} \sum\limits_{|\bm{\ell}|=2 k} T_{A}(\bm{\ell})}{\sum\limits_{k=0}^{k_{0}} \sum\limits_{|\bm{\ell}|=2 k} T(\bm{\ell})},
	\end{equation}
	and thus
	\begin{align}\label{p-14}
	| I-I_{2k_{0}} |
	&= \Bigg|\frac{\sum\limits_{k=0}^{\lfloor\frac{N}{2}\rfloor}\sum\limits_{|\bm{\ell}|=2k}T_{A}(\bm{\ell})}{\sum\limits_{k=0}^{\lfloor\frac{N}{2}\rfloor}\sum\limits_{|\bm{\ell}|=2k}T(\bm{\ell})}-\frac{\sum\limits_{k=0}^{k_{0}}\sum\limits_{|\bm{\ell}|=2k}T_{A}(\bm{\ell})}{\sum\limits_{k=0}^{k_{0}}\sum\limits_{|\bm{\ell}|=2k}T(\bm{\ell})}\Bigg| \notag \\
	&\leq \Bigg|\frac{\sum\limits_{k=0}^{\lfloor\frac{N}{2}\rfloor}\sum\limits_{|\bm{\ell}|=2k}T_{A}(\bm{\ell})}{\sum\limits_{k=0}^{\lfloor\frac{N}{2}\rfloor}\sum\limits_{|\bm{\ell}|=2k}T(\bm{\ell})}-\frac{\sum\limits_{k=0}^{k_{0}}\sum\limits_{|\bm{\ell}|=2k}T_{A}(\bm{\ell})}{\sum\limits_{k=0}^{\lfloor\frac{N}{2}\rfloor}\sum\limits_{|\bm{\ell}|=2k}T(\bm{\ell})}\Bigg|
	+\Bigg|\frac{\sum\limits_{k=0}^{k_{0}}\sum\limits_{|\bm{\ell}|=2k}T_{A}(\bm{\ell})}{\sum\limits_{k=0}^{\lfloor\frac{N}{2}\rfloor}\sum\limits_{|\bm{\ell}|=2k}T(\bm{\ell})}-\frac{\sum\limits_{k=0}^{k_{0}}\sum\limits_{|\bm{\ell}|=2k}T_{A}(\bm{\ell})}{\sum\limits_{k=0}^{k_{0}}\sum\limits_{|\bm{\ell}|=2k}T(\bm{\ell})}\Bigg|\notag \\
	&= \Bigg| \frac{\sum\limits_{k=k_{0}+1}^{\lfloor\frac{N}{2}\rfloor}\sum\limits_{|\bm{\ell}|=2k}T_{A}(\bm{\ell})}{\sum\limits_{k=0}^{\lfloor\frac{N}{2}\rfloor}\sum\limits_{|\bm{\ell}|=2k}T(\bm{\ell})}\Bigg|
	+\Bigg|\frac{\sum\limits_{k=0}^{k_{0}}\sum\limits_{|\bm{\ell}|=2k}T_{A}(\bm{\ell})}{\sum\limits_{k=0}^{\lfloor\frac{N}{2}\rfloor}\sum\limits_{|\bm{\ell}|=2k}T(\bm{\ell})}\Bigg| \Bigg|\frac{\sum\limits_{k=k_{0}+1}^{\lfloor\frac{N}{2}\rfloor}\sum\limits_{|\bm{\ell}|=2k}T(\bm{\ell})}{\sum\limits_{k=0}^{k_{0}}\sum\limits_{|\bm{\ell}|=2k}T(\bm{\ell})}\Bigg|\notag \\
	&= I_{A}+I_{B}I_{C},
	\end{align}
	where
	\begin{equation*}
	I_{A}=\Bigg| \frac{\sum\limits_{k=k_{0}+1}^{\lfloor\frac{N}{2}\rfloor}\sum\limits_{|\bm{\ell}|=2k}T_{A}(\bm{\ell})}{\sum\limits_{k=0}^{\lfloor\frac{N}{2}\rfloor}\sum\limits_{|\bm{\ell}|=2k}T(\bm{\ell})}\Bigg|,\quad I_{B}=\Bigg|\frac{\sum\limits_{k=0}^{k_{0}}\sum\limits_{|\bm{\ell}|=2k}T_{A}(\bm{\ell})}{\sum\limits_{k=0}^{\lfloor\frac{N}{2}\rfloor}\sum\limits_{|\bm{\ell}|=2k}T(\bm{\ell})}\Bigg|\quad \text{and}\quad
	I_{C}=\Bigg|\frac{\sum\limits_{k=k_{0}+1}^{\lfloor\frac{N}{2}\rfloor}\sum\limits_{|\bm{\ell}|=2k}T(\bm{\ell})}{\sum\limits_{k=0}^{k_{0}}\sum\limits_{|\bm{\ell}|=2k}T(\bm{\ell})}\Bigg|.
	\end{equation*}
	We notice when $\bm{\ell}$ contains $2k$ kinks, using \eqref{p-1}, \eqref{p-8} and \eqref{p-9} we have
	
	\begin{align}\label{p-15}
	|T_{A}(\bm{\ell})|
	&\le |T_{B}(\bm{\ell})|+\frac{1}{N} \sum_{n=1}^{N} |T_{C, n}(\bm{\ell})| \notag\\
	&\le C_{3} M^{N}C_{2}^{N} \beta_{N}^{2 k} C(\bm{\ell})+C_{3} M^{N}C_{2}^{N} N \beta_{N}^{2 k} \left(\frac{1}{N}\sum_{n=1}^{N}C_{n}(\bm{\ell})\right) \notag\\
	&\le C_{3} M^{N}C_{2}^{N} N \beta_{N}^{2 k} \left(C(\bm{\ell})+\frac{1}{N}\sum_{n=1}^{N}C_{n}(\bm{\ell})\right).
	\end{align}
	Using \eqref{p-11}, \eqref{p-12} and \eqref{p-15}, we have
	\begin{align}\label{p-16}
	\frac{|T_{A}(\bm{\ell})|}{C(\bm{\ell}_{0})}
	&\le C_{3} M^{N}C_{2}^{N} N \beta_{N}^{2 k} \left(\frac{C(\bm{\ell})}{C(\bm{\ell}_{0})}+\frac{1}{N}\sum_{n=1}^{N}\frac{C_{n}(\bm{\ell})}{C(\bm{\ell}_{0})}\right)\notag\\
	&\le C_{3} M^{N}C_{2}^{N} N \beta_{N}^{2 k} \left(e^{C_{1}}+\frac{1}{N}\sum_{n=1}^{N}e^{C_{1}}\right)\notag\\
	&=2C_{3}e^{C_{1}} M^{N}C_{2}^{N} N \beta_{N}^{2 k}.
	\end{align}
	Since $N^{-2k}\tbinom{N}{2k}\leq \frac{1}{(2k)!}$ and $\beta_{N}=\frac{1}{N}$, using \eqref{p-16} and Lemma \ref{lem-1} we have
	\begin{align}\label{p-17}
	\frac{\sum\limits_{|\bm{\ell}|=2 k}|T_{A}(\bm{\ell})|}{C(\bm{\ell}_{0})}
	&\le \sum\limits_{|\bm{\ell}|=2 k}2C_{3}e^{C_{1}} M^{N}C_{2}^{N} N \beta_{N}^{2 k}\notag\\
	&=4C_{3}e^{C_{1}} M^{N}C_{2}^{N} N \beta_{N}^{2 k}\tbinom{N}{2k} \notag\\
	&\le 4C_{3}e^{C_{1}} M^{N}C_{2}^{N} N \frac{1}{(2k)!}
	\end{align}
	Then we can use \eqref{p-7} and \eqref{p-17} to bound $I_{A}$ and $I_{B}$ respectively,
	\begin{equation}\label{p-18}
	I_{A}\le\frac{\left|\sum\limits_{k=k_{0}+1}^{\lfloor\frac{N}{2}\rfloor}\sum\limits_{|\bm{\ell}|=2k}T_{A}(\bm{\ell})\right|}{C(\bm{\ell}_{0})} \le \sum\limits_{k=k_{0}+1}^{\lfloor\frac{N}{2}\rfloor}
	\frac{\sum\limits_{|\bm{\ell}|=2 k}|T_{A}(\bm{\ell})|}{C(\bm{\ell}_{0})}
	\le 4C_{3}e^{C_{1}} M^{N}C_{2}^{N} N \sum\limits_{k=k_{0}+1}^{\lfloor\frac{N}{2}\rfloor}\frac{1}{(2k)!},
	\end{equation}
	\begin{equation}\label{p-19}
	I_{B}\le \frac{\left|\sum\limits_{k=0}^{k_{0}}\sum\limits_{|\bm{\ell}|=2k}T_{A}(\bm{\ell})\right|}{C(\bm{\ell}_{0})}\le
	\sum\limits_{k=0}^{k_{0}}
	\frac{\sum\limits_{|\bm{\ell}|=2 k}|T_{A}(\bm{\ell})|}{C(\bm{\ell}_{0})}
	\le 4C_{3}e^{C_{1}} M^{N}C_{2}^{N} N \sum\limits_{k=0}^{k_{0}}\frac{1}{(2k)!}.
	\end{equation}
	Searching a bound for $I_{C}$ is much easier. Assume $\bm{\ell}$ contains $2k$ kinks, using \eqref{p-5} and \eqref{p-11} we have
	\begin{equation}\label{p-20}
	\frac{T(\bm{\ell})}{C(\bm{\ell}_{0})}\le M^{N}C_{2}^{N} \beta_{N}^{2 k} \frac{C(\bm{\ell})}{C(\bm{\ell}_{0})}\le e^{C_{1}}M^{N}C_{2}^{N} \beta_{N}^{2 k}.
	\end{equation}
	Since $N^{-2k}\tbinom{N}{2k}\leq \frac{1}{(2k)!}$ and $\beta_{N}=\frac{1}{N}$, using \eqref{p-20} and Lemma \ref{lem-1} we have
	\begin{equation}\label{p-21}
	\frac{\sum\limits_{|\bm{\ell}|=2 k}T(\bm{\ell})}{C(\bm{\ell}_{0})} \le 2e^{C_{1}}M^{N}C_{2}^{N} \beta_{N}^{2 k}\tbinom{N}{2k}\le 2e^{C_{1}}M^{N}C_{2}^{N}\frac{1}{(2k)!}.
	\end{equation}
	Then use \eqref{p-7} and \eqref{p-21} to bound $I_{C}$,
	\begin{equation}\label{p-22}
	I_{C}\le \frac{\left|\sum\limits_{k=k_{0}+1}^{\lfloor\frac{N}{2}\rfloor}\sum\limits_{|\bm{\ell}|=2k}T(\bm{\ell})\right|}{C(\bm{\ell}_{0})}\le \sum\limits_{k=k_{0}+1}^{\lfloor\frac{N}{2}\rfloor}
	\frac{\sum\limits_{|\bm{\ell}|=2 k}|T(\bm{\ell})|}{C(\bm{\ell}_{0})}\le 2e^{C_{1}}M^{N}C_{2}^{N}\sum\limits_{k=k_{0}+1}^{\lfloor\frac{N}{2}\rfloor}\frac{1}{(2k)!}.
	\end{equation}
	To get our conclusion, we use \eqref{p-18}, \eqref{p-19} and \eqref{p-22} to bound $|I-I_{2k_{0}}|$ according to \eqref{p-14}
	\begin{align}
	|I-I_{2k_{0}}|
	&\le I_{A}+I_{B}I_{C}\notag\\
	&\le 4C_{3}e^{C_{1}} M^{N}C_{2}^{N} N \sum\limits_{k=k_{0}+1}^{\lfloor\frac{N}{2}\rfloor}\frac{1}{(2k)!}+4C_{3}e^{C_{1}} M^{N}C_{2}^{N} N \sum\limits_{k=0}^{k_{0}}\frac{1}{(2k)!}\times\notag\\
	&\times 2e^{C_{1}}M^{N}C_{2}^{N}\sum\limits_{k=k_{0}+1}^{\lfloor\frac{N}{2}\rfloor}\frac{1}{(2k)!}\notag \\
	&=4C_{3}e^{C_{1}} M^{N}C_{2}^{N} N \sum\limits_{k=k_{0}+1}^{\lfloor\frac{N}{2}\rfloor}\frac{1}{(2k)!}+\notag \\
	&+8C_{3}e^{2C_{1}} M^{2N}C_{2}^{2N} N \left(\sum\limits_{k=0}^{k_{0}}\frac{1}{(2k)!}\right)\left(\sum\limits_{k=k_{0}+1}^{\lfloor\frac{N}{2}\rfloor}\frac{1}{(2k)!}\right).\notag
	\end{align}
	We notice $4C_{3}e^{C_{1}} M^{N}C_{2}^{N} N \le 8C_{3}e^{2C_{1}} M^{2N}C_{2}^{2N} N \left(\sum\limits_{k=0}^{k_{0}}\frac{1}{(2k)!}\right)$, and thus
	\begin{align}\label{p-23}
	|I-I_{2k_{0}}|
	&\le 16C_{3}e^{2C_{1}} M^{2N}C_{2}^{2N} N \left(\sum\limits_{k=0}^{k_{0}}\frac{1}{(2k)!}\right)\left(\sum\limits_{k=k_{0}+1}^{\lfloor\frac{N}{2}\rfloor}\frac{1}{(2k)!}\right).
	\end{align}
	According to \eqref{p-23}, let $C_{4}=16C_{3}e^{2C_{1}}\sum\limits_{k=0}^{+\infty}\frac{1}{(2k)!}$ and $C_{5}=M^{2}C_{2}^{2}$. We complete our proof. $\hfill\square$
\end{proof}

\begin{rem}
We notice the Theorem \ref{thm-1} is established only when we assume the difference of the two energy surface is bounded from the above side, $V_{00}(q)-V_{11}(q)\le C_{1}$, which contains the special case $|V_{00}(q)-V_{11}(q)|\le C_{1}$. It is because the right side of \eqref{p-10} only contains $V_{00}(q)-V_{11}(q)$ but not the opposite term $V_{11}(q)-V_{00}(q)$. Intuitively, when we assume $V_{00}(q)-V_{11}(q)\le C_{1}$, the contribution of configurations with index sequence $\bm{\ell}_{0}=\{0,\cdots,0\}$ can dominate that of other configurations even those with index sequence $\bm{\ell}_{1}=\{1,\cdots,1\}$, and thus becomes a dominant part to the thermal average. 
\end{rem}
\begin{rem}
When we change the assumption $V_{00}(q)-V_{11}(q)<C_{1}$ to $V_{11}(q)-V_{00}(q)<C_{1}$ and keep other assumptions in Theorem \ref{thm-1}, we can still prove the same result only by changing $\bm{\ell}_{0}$ to $\bm{\ell}_{1}=\{1,\cdots,1\}$ in the proof. 
\end{rem}
\begin{rem}
When we assume the observable $\widehat{A}(\hat{q})$ is diagonal, which means $A_{01}(\hat{q})=A_{10}(\hat{q})=0$, and keep other assumptions. Revise the above proof and we notice $T_{C,k}(\bm{\ell})$ disappears because $A_{01}=A_{10}=0$. Using the same method in the proof, we can show $$\left|I-I_{2 k_{0}}\right| \leq C_{4} C_{5}^{N}\sum\limits
_{k=k_{0}+1}^{\left\lfloor\frac{N}{2}\right\rfloor} \frac{1}{(2 k) !},$$ where the factor $N$ in the previous result disappears.
\end{rem}

\par Besides the extended ring polymer representation and PIMD-SH method that we have been discussing so far, in recent years some other methods have been proposed to calculate the thermal average in the nonadiabatic regime. In \cite{LiuPath},  a proper reference measure was introduced for thermal averages in the nonadiabatic regime, and the weighted functions implicitly include all surface index configurations. In \cite{tao2018path}, an isomorphic Hamiltionian was introduced for multi-electronic-state quantum systems where the nonadiabtic coupling was included in the off-diagonal elements of the reduced matrix representation. The methods in \cite{LiuPath,tao2018path} do not yet make use of the asymptotic decreasing property with respect to the surface index sequence level which has been discussion in this section and the contributions of all surface index sequences are integrated together. However, the summation over the surface index configurations with exponentially small contributions to the thermal average causes unnecessary waste of computational cost and influences the efficiency of numerical methods. Therefore, the proper truncation idea can in theory apply to those methods as well, whereas there might be additional challenges which we may explore in the future.

\par Although Theorem \ref{thm-1} does not directly imply an alternative numerical method, we shall see in Section \ref{sec-4} that it provides a practical guide to further enhance the efficiency of the improved numerical algorithms with minimal truncation error introduced.

\section{Multi-level Monte Carlo path integral molecular dynamics method}\label{sec-4}

\subsection{An introduction of a proper reference measure}
\hspace{4mm} Let $\tilde{\pi}$ denote a distribution on $\mathbb{R}^{2d N}$ (phase space without surface indexes) which is to be specified, and we call it reference measure. We notice from \eqref{I_2}
\begin{align*}
I& =\frac{\int_{\mathbb{R}^{2dN}}\sum\limits_{k=0}^{\lfloor\frac{N}{2}\rfloor}\sum\limits_{|\bm{\ell}|=2k}W_{N}[A](\bm{q},\bm{p},\bm{\ell})e^{-\beta_{N}H_{N}(\bm{q},\bm{p},\bm{\ell})}\mathrm{d}\bm{q}\mathrm{d}\bm{p}}{\int_{\mathbb{R}^{2dN}}\sum\limits_{k=0}^{\lfloor\frac{N}{2}\rfloor}\sum\limits_{|\bm{\ell}|=2k}e^{-\beta_{N}H_{N}(\bm{q},\bm{p},\bm{\ell})}\mathrm{d}\bm{q}\mathrm{d}\bm{p}}  \\
&=\frac{\int_{\mathbb{R}^{2dN}}\left(\sum\limits_{k=0}^{\lfloor\frac{N}{2}\rfloor}\sum\limits_{|\bm{\ell}|=2k}W_{N}[A](\bm{q},\bm{p},\bm{\ell})\frac{e^{-\beta_{N}H_{N}(\bm{q},\bm{p},\bm{\ell})}}{\tilde{\pi}(\bm{q},\bm{p})}\right)\tilde{\pi}(\bm{q},\bm{p})\mathrm{d}\bm{q}\mathrm{d}\bm{p}}{\int_{\mathbb{R}^{2dN}}\left(\sum\limits_{k=0}^{\lfloor\frac{N}{2}\rfloor}\sum\limits_{|\bm{\ell}|=2k}\frac{e^{-\beta_{N}H_{N}(\bm{q},\bm{p},\bm{\ell})}}{\tilde{\pi}(\bm{q},\bm{p})}\right)\tilde{\pi}(\bm{q},\bm{p})\mathrm{d}\bm{q}\mathrm{d}\bm{p}},   
\end{align*}
which means $I$ can be reformulated to the ratio of two expectations with respect to $\tilde{\pi}$. In the following, we let $\tilde{\pi}$ take the form
\begin{equation}\label{4-1}
    \tilde{\pi}(\bm{q},\bm{p}):=\frac{e^{-\beta_{N}H_{N}(\bm{q},\bm{p},\bm{\ell}_{0})}}{\int_{\mathbb{R}^{2dN}}e^{-\beta_{N}H_{N}(\bm{q},\bm{p},\bm{\ell}_{0})}\mathrm{d}\bm{q}\mathrm{d}\bm{p}},
\end{equation}
where we recall here for convenience $\bm{\ell}_{0}=\{0,\cdots,0\}$. The motivation for choosing such a $\tilde{\pi}$ will be elaborated below. We remark that other choices of proper reference measures $\tilde{\pi}$ are possible, see, e.g. \cite{LiuPath}.
\par We introduce the following notations. We use $\mathrm{E}_{\tilde{\pi}}A_{k}$ and $\mathrm{E}_{\tilde{\pi}}B_{k}$ to represent the expectation of $$A_{k}(\bm{q},\bm{p}):=\sum_{|l|=2 k} \frac{W_{N}[A](\bm{q},\bm{p},\bm{\ell}) e^{-\beta_{N} H_{N}(\bm{q},\bm{p},\bm{\ell})}}{e^{-\beta_{N} H_{N}\left(\bm{q},\bm{p},\bm{\ell}_{0}\right)}} \quad\text{and}\quad B_{k}(\bm{q},\bm{p}):=\sum_{|l|=2 k} \frac{ e^{-\beta_{N} H_{N}(\bm{q},\bm{p},\bm{\ell})}}{e^{-\beta_{N} H_{N}\left(\bm{q},\bm{p},\bm{\ell}_{0}\right)}}$$ on $\mathbb{R}^{2dN}$ with respect to distribution $\tilde{\pi}(\bm{q},\bm{p})$ respectively, namely,
\begin{align*}
    &\mathrm{E}_{\tilde{\pi}}A_{k}:=\int_{R^{2 N d}}A_{k}(\bm{q},\bm{p}) \tilde{\pi}(\bm{q},\bm{p}) \mathrm{d} \bm{q} \mathrm{d} \bm{p}=\int_{R^{2 N d}}\left(\sum_{|l|=2 k} \frac{W_{N}[A](\bm{q},\bm{p},\bm{\ell}) e^{-\beta_{N} H_{N}(\bm{q},\bm{p},\bm{\ell})}}{e^{-\beta_{N} H_{N}\left(\bm{q},\bm{p},\bm{\ell}_{0}\right)}}\right) \tilde{\pi}(\bm{q},\bm{p}) \mathrm{d} \bm{q} \mathrm{d} \bm{p}\\
    &\text{and} \quad \mathrm{E}_{\tilde{\pi}}B_{k}:=\int_{R^{2 N d}}B_{k}(\bm{q},\bm{p}) \tilde{\pi}(\bm{q},\bm{p}) \mathrm{d} \bm{q} \mathrm{d} \bm{p}=\int_{R^{2 N d}}\left(\sum_{|l|=2 k} \frac{ e^{-\beta_{N} H_{N}(\bm{q},\bm{p},\bm{\ell})}}{e^{-\beta_{N} H_{N}\left(\bm{q},\bm{p},\bm{\ell}_{0}\right)}}\right) \tilde{\pi}(\bm{q},\bm{p}) \mathrm{d} \bm{q} \mathrm{d} \bm{p}.
\end{align*}
Then the extended ring polymer representation $I$ can be transformed to the ratio of two expectations with respect to the distribution $\tilde{\pi}$:
\begin{equation}\label{4-2}
    I
    =\frac{\int_{\mathbb{R}^{2dN}}\sum\limits_{k=0}^{\lfloor\frac{N}{2}\rfloor}\sum\limits_{|\bm{\ell}|=2k}W_{N}[A](\bm{q},\bm{p},\bm{\ell})e^{-\beta_{N}H_{N}(\bm{q},\bm{p},\bm{\ell})}\mathrm{d}\bm{q}\mathrm{d}\bm{p}}{\int_{\mathbb{R}^{2dN}}\sum\limits_{k=0}^{\lfloor\frac{N}{2}\rfloor}\sum\limits_{|\bm{\ell}|=2k}e^{-\beta_{N}H_{N}(\bm{q},\bm{p},\bm{\ell})}\mathrm{d}\bm{q}\mathrm{d}\bm{p}}  
    =\frac{\mathrm{E}_{\tilde{\pi}}\sum\limits_{k=0}^{\lfloor\frac{N}{2}\rfloor}A_{k}}{\mathrm{E}_{\tilde{\pi}}\sum\limits_{k=0}^{\lfloor\frac{N}{2}\rfloor}B_{k}}.
\end{equation}
\par Our motivation for choosing such a specific distribution $\tilde{\pi}$ lies in two aspects. First, according to the expression of $W_{N}[A]$ in \eqref{W} and $\langle \ell| G_{k}|\ell^{\prime}\rangle$ in \eqref{h}, we notice
\begin{equation*}
    e^{\beta_{N}\left\langle \ell_{k}\left|G_{k}\right| \ell_{k+1}\right\rangle}=
    \begin{cases}
    e^{\frac{\beta_{N}}{2 M} p_{k}^{2}+\frac{M}{2 \beta_{N}} \left(q_{k}-q_{k+1}\right)^{2}+\beta_{N} V_{\ell_{k}\ell_{k}}(q_{k})}\cosh^{-1}(\beta_{N}|V_{01}(q_{k})|),&\ell_{k}=\ell_{k+1}\\
    e^{\frac{\beta_{N}}{2 M} p_{k}^{2}+\frac{M}{2 \beta_{N}} \left(q_{k}-q_{k+1}\right)^{2}+\beta_{N} \frac{V_{00}(q_{k})+V_{11}(q_{k})}{2}}\sinh^{-1}(\beta_{N}|V_{01}(q_{k})|),&\ell_{k}\neq \ell_{k+1}
    \end{cases}
\end{equation*}
which means if there exists a kink between $\ell_{k}$ and $\ell_{k+1}$, because of the existence of $\sinh^{-1}(\beta_{N}|V_{01}(q_{k})|)\approx \sinh^{-1}(\frac{C}{N})$, the term $e^{\beta_{N}\left\langle \ell_{k}\left|G_{k}\right| \ell_{k+1}\right\rangle}$ attains a large value. We notice from \eqref{W} that $W_{N}[A]$ depends linearly on $e^{\beta_{N}\left\langle \ell_{k}\left|G_{k}\right| \ell_{k+1}\right\rangle}$ and hence, when there exists a kink between $\ell_{k}$ and $\ell_{k+1}$, $W_{N}[A]$ also attains a large value. As a consequence, when we use PIMD-SH method to calculate the ring polymer representation $I$, the trajectory may occasionally visit large values, affecting its numerical performance. We can also see this numerical performance from Figure 7 of Section \ref{sec-5} where the trajectory of PIMD-SH method visits some large values in a non-negligible probability. However, we can avoid this shortcoming when we consider using \eqref{4-2} to calculate $I$, because we notice the term $e^{\beta_{N}\left\langle \ell_{k}\left|G_{k}\right| \ell_{k+1}\right\rangle}$ in $W_{N}[A]$ is neutralized by the term $e^{-\beta_{N}\left\langle \ell_{k}\left|G_{k}\right| \ell_{k+1}\right\rangle}$ in $e^{-\beta_{N}H_{N}(\bm{q},\bm{p},\bm{\ell})}$, which means $W_{N}[A](\bm{q},\bm{p},\bm{\ell})e^{-\beta_{N}H_{N}(\bm{q},\bm{p},\bm{\ell})}$ and thus $A_{k}$ and $B_{k}$ won't take large values.
\par Second, according to the formula of $e^{-\beta_{N} H_{N}(\bm{q},\bm{p},\bm{\ell})}$, when $\bm{\ell}$ contains $2k$ kinks, $e^{-\beta_{N} H_{N}(\bm{q},\bm{p},\bm{\ell})}\approx \frac{C}{N^{2k}}$, exponentially small compared to $e^{-\beta_{N} H_{N}(\bm{q},\bm{p},\bm{\ell}_{0})}$. Because $A_{k}$ and $B_{k}$ depend linearly on $e^{-\beta_{N} H_{N}(\bm{q},\bm{p},\bm{\ell})}$ where $|\bm{\ell}|=2k$, $A_{k}$ and $B_{k}$ are very small when $k$ is large. And we notice
\begin{equation*}
B_{k}(\bm{q},\bm{p})=\sum_{|l|=2 k} \frac{ e^{-\beta_{N} H_{N}(\bm{q},\bm{p},\bm{\ell})}}{e^{-\beta_{N} H_{N}\left(\bm{q},\bm{p},\bm{\ell}_{0}\right)}} \approx
2\tbinom{N}{2k}\frac{C}{N^{2k}}\le \frac{C}{(2k)!},
\end{equation*}
which means not only $A_{k}$ and $B_{k}$ are very small but also decrease fast while $k$ grows. An advantage of the fast decreasing property of $A_{k}$ and $B_{k}$ is that we can  optimize the sampling times for each sub-estimators $\widehat{\mathrm{E}}_{\tilde{\pi}}A_{k}$ and $\widehat{\mathrm{E}}_{\tilde{\pi}}B_{k}$ defined later to minimize the total computational variance, leveraging the structure of the system itself.
\par Recall the definition of $I_{2k_{0}}$ and the notations we introduced above, we have
\begin{equation*}
I_{2k_{0}}
=\frac{\int_{\mathbb{R}^{2dN}}\sum\limits_{k=0}^{k_{0}}\sum\limits_{|\bm{\ell}|=2k}W_{N}[A](\bm{q},\bm{p},\bm{\ell})e^{-\beta_{N}H_{N}(\bm{q},\bm{p},\bm{\ell})}\mathrm{d}\bm{q}\mathrm{d}\bm{p}}{\int_{\mathbb{R}^{2dN}}\sum\limits_{k=0}^{k_{0}}\sum\limits_{|\bm{\ell}|=2k}e^{-\beta_{N}H_{N}(\bm{q},\bm{p},\bm{\ell})}\mathrm{d}\bm{q}\mathrm{d}\bm{p}}
=\frac{\mathrm{E}_{\tilde{\pi}}\sum\limits_{k=0}^{k_{0}}A_{k}}{\mathrm{E}_{\tilde{\pi}}\sum\limits_{k=0}^{k_{0}}B_{k}}
=\frac{\sum\limits_{k=0}^{k_{0}}\mathrm{E}_{\tilde{\pi}}A_{k}}{\sum\limits_{k=0}^{k_{0}}\mathrm{E}_{\tilde{\pi}}B_{k}}
. 
\end{equation*}
Theorem \ref{thm-1} guarantees $I_{2k_{0}}$ is a good approximation to $I$ when $k_{0}$ is properly chosen. Lemma \ref{lem-1} shows each $A_{k}$ and $B_{k}$ is a summation of $2\tbinom{N}{2k}$ terms, which means a huge computational cost when $k$ is large. However, a truncation by $k_{0}$ means we only need to compute $A_{k}$ and $B_{k}$ for relatively small $k$, which saves the computation power. Instead computing $\mathrm{E}_{\tilde{\pi}}\sum\limits_{k=0}^{k_{0}}A_{k}$ and $\mathrm{E}_{\tilde{\pi}}\sum\limits_{k=0}^{k_{0}}B_{k}$, we compute the sub-estimators $\mathrm{E}_{\tilde{\pi}}A_{k}\,(0\le k\le k_{0})$ and $\mathrm{E}_{\tilde{\pi}}B_{k}\,(0\le k\le k_{0})$ respectively and optimize the computation power assigned to each sub-estimator. 
%\par According to the right hand of \eqref{4-2}, we can compute the two expectations $\mathrm{E}_{\tilde{\pi}}\sum\limits_{k=0}^{\lfloor\frac{N}{2}\rfloor}A_{k}$ and $\mathrm{E}_{\tilde{\pi}}\sum\limits_{k=0}^{\lfloor\frac{N}{2}\rfloor}B_{k}$ respectively and let the ratio of them to approximate $I$. But to get $\sum\limits_{k=0}^{\lfloor\frac{N}{2}\rfloor}A_{k}$ or $\sum\limits_{k=0}^{\lfloor\frac{N}{2}\rfloor}B_{k}$, we need to sum $\sum\limits_{k=0}^{\lfloor\frac{N}{2}\rfloor}2\tbinom{N}{2k}$ times because each $A_{k}$ and $B_{k}$ is a summation of $2\tbinom{N}{2k}$ terms, which means a huge computational cost. Moreover, while $k$ increases, $A_{k}$ and $B_{k}$ decrease fast and become negligible compared with $\sum\limits_{k=0}^{\lfloor\frac{N}{2}\rfloor}A_{k}$ and $\sum\limits_{k=0}^{\lfloor\frac{N}{2}\rfloor}B_{k}$. 
%Thus we use $\sum\limits_{k=0}^{k_{0}}A_{k}$ and $\sum\limits_{k=0}^{k_{0}}B_{k}$ to approximate $\sum\limits_{k=0}^{\lfloor\frac{N}{2}\rfloor}A_{k}$ and $\sum\limits_{k=0}^{\lfloor\frac{N}{2}\rfloor}B_{k}$ with a properly chosen $k_{0}$ and instead of computing $\mathrm{E}_{\tilde{\pi}}\sum\limits_{k=0}^{k_{0}}A_{k}$ and $\mathrm{E}_{\tilde{\pi}}\sum\limits_{k=0}^{k_{0}}B_{k}$, we use sub-estimators to compute $\mathrm{E}_{\tilde{\pi}}A_{k}\,(0\le k\le k_{0})$ and $\mathrm{E}_{\tilde{\pi}}B_{k}\,(0\le k\le k_{0})$ respectively and sum them up to get two expectations. Recall the definition of $I_{2k_{0}}$, we have
\par Notice the similarity of the representations for $\mathrm{E}_{\tilde{\pi}}A_{k}$ and $\mathrm{E}_{\tilde{\pi}}B_{k}$, we need to find a method to calculate the integral
$$\int_{\mathbb{R}^{2dN}}f(\bm{q},\bm{p})\tilde{\pi}(\bm{q},\bm{p})\mathrm{d}\bm{q}\mathrm{d}\bm{p}$$ for a given function $f(\bm{q},\bm{p})$ with respect to the distribution $\tilde{\pi}(\bm{q},\bm{p})$ on $\mathbb{R}^{2dN}$.
For example, the Langevin dynamics $\bm{z}(t)$ below is ergodic with respect to $\tilde{\pi}$ 
\begin{equation}\label{4-3}
\begin{cases}
\mathrm{d} \bm{q}=\nabla_{\bm{p}} H_{N}(\bm{q}(t), \bm{p}(t), \bm{\ell}_{0}) \mathrm{d} t,\\
\mathrm{d} \bm{p}=-\nabla_{\bm{q}} H_{N}(\bm{q}(t), \bm{p}(t), \bm{\ell}_{0}) \mathrm{d} t-\gamma \bm{p} \mathrm{d} t+\sqrt{2 \gamma \beta_{N}^{-1} M} \mathrm{d} \bm{B},
\end{cases}
\end{equation}
where $\gamma$ denotes the friction constant. Checking the ergodic property is a well studied subject and we omit the details here. Thus to compute the integral $\mathrm{E}_{\tilde{\pi}}A_{k}\,(0\le k\le k_{0})$ and $\mathrm{E}_{\tilde{\pi}}B_{k}\,(0\le k\le k_{0})$ to get the truncated thermal average $I_{2k_{0}}$, we sample the trajectory $\bm{z}(t)$ by a time average
\begin{equation*}
    \mathrm{E}_{\tilde{\pi}}A_{k}\approx\lim _{T \rightarrow \infty} \frac{1}{T} \int_{0}^{T} A_{k}(\bm{z}(t)) \mathrm{d} t
    \quad \text{and}\quad
    \mathrm{E}_{\tilde{\pi}}B_{k}\approx\lim _{T \rightarrow \infty} \frac{1}{T} \int_{0}^{T} B_{k}(\bm{z}(t)) \mathrm{d} t.
\end{equation*}
\par Next we consider the numerical implementation, the sub-estimators $\widehat{\mathrm{E}}_{\tilde{\pi}}A_{k}$ and $\widehat{\mathrm{E}}_{\tilde{\pi}}B_{k}$ for computing $\mathrm{E}_{\tilde{\pi}}A_{k}$ and $\mathrm{E}_{\tilde{\pi}}B_{k}$ can be written as
\begin{equation}\label{4-4}
    \mathrm{E}_{\tilde{\pi}}A_{k} \approx \widehat{\mathrm{E}}_{\tilde{\pi}}A_{k}(N_{k}):=\frac{1}{N_{k}}\sum\limits_{i=1}^{N_{k}}A_{k}(\bm{z}_{k}^{a}(t_{i}))
    \quad\text{and}\quad
    \mathrm{E}_{\tilde{\pi}}B_{k} \approx \widehat{\mathrm{E}}_{\tilde{\pi}}B_{k}(N_{k}):=\frac{1}{N_{k}}\sum\limits_{i=1}^{N_{k}}B_{k}(\bm{z}_{k}^{b}(t_{i})),
\end{equation}
where $\bm{z}^{a}_{k}(t)$ and $\bm{z}_{k}^{b}(t)$ are independent trajectories of $\bm{z}(t)$ constructed in \eqref{4-3} and $t_{i}=i\Delta t$. To sum up, we respectively sample the trajectory $\bm{z}_{k}^{a}$ and $\bm{z}_{k}^{b}$ with a time step $\Delta t$ and the sampling times $N_{k}$ to approximate $\mathrm{E}_{\tilde{\pi}}A_{k}$ and $\mathrm{E}_{\tilde{\pi}}B_{k}$.
\iffalse
\par When we consider numerical scheme, the sub-estimators $\widehat{\mathrm{E}}_{\tilde{\pi}}A_{k}$ and $\widehat{\mathrm{E}}_{\tilde{\pi}}B_{k}$ for computing $\mathrm{E}_{\tilde{\pi}}A_{k}$ and $\mathrm{E}_{\tilde{\pi}}B_{k}$ can be written as
\begin{equation}
    \mathrm{E}_{\tilde{\pi}}A_{k} \approx \widehat{\mathrm{E}}_{\tilde{\pi}}A_{k}(N_{k}^{a}):=\frac{1}{N_{k}^{a}}\sum_{i=0}^{N_{k}^{a}}A_{k}(\bm{z}_{k}^{a}(t_{k,i}^{a}))
    \quad\text{and}\quad
    \mathrm{E}_{\tilde{\pi}}B_{k} \approx \widehat{\mathrm{E}}_{\tilde{\pi}}B_{k}(N_{k}^{b}):=\frac{1}{N_{k}^{b}}\sum_{i=0}^{N_{k}^{b}}B_{k}(\bm{z}_{k}^{b}(t_{k,i}^{b})),
\end{equation}
where $\bm{z}_{k}^{a}(t)$ and $\bm{z}_{k}^{b}(t)$ are independent copies of trajectory $\bm{z}(t)$ constructed in (38), which means we sample the trajectory $\bm{z}_{k}^{a}(t)$ with $N_{k}^{a}$ times at $t=t^{a}_{k,i}\,(1\le i \le N_{k}^{a})$ to get $\widehat{\mathrm{E}}_{\tilde{\pi}}A_{k}(N_{k}^{a})$ and sample the trajectory $\bm{z}_{k}^{b}(t)$ with $N_{k}^{b}$ times at $t=t^{b}_{k,i}\,(1\le i \le N_{k}^{b})$ to get $\widehat{\mathrm{E}}_{\tilde{\pi}}B_{k}(N_{k}^{b})$. For the convenience of our numerical method, we set $t^{a}_{k,i}=t^{b}_{k,i}=i\Delta t$, which means the times steps are the same in every sub-estimators, and set $N_{k}^{a}=N_{k}^{b}=N_{k}$, which means for $\widehat{\mathrm{E}}_{\tilde{\pi}}A_{k}$ and $\widehat{\mathrm{E}}_{\tilde{\pi}}B_{k}$ the computation times are the same.
\fi
\par Now we have the the numerical approximation of the truncated thermal average $I_{2k_{0}}$:
\begin{equation}\label{4-5}
I_{2k_{0}}=\frac{\sum\limits_{k=0}^{k_{0}}\mathrm{E}_{\tilde{\pi}}A_{k}}{\sum\limits_{k=0}^{k_{0}}\mathrm{E}_{\tilde{\pi}}B_{k}}
\approx \frac{\sum\limits_{k=0}^{k_{0}}\widehat{\mathrm{E}}_{\tilde{\pi}}A_{k}(N_{k})}{\sum\limits_{k=0}^{k_{0}}\widehat{\mathrm{E}}_{\tilde{\pi}}B_{k}(N_{k})}=\frac{\sum\limits_{k=0}^{k_{0}}\frac{1}{N_{k}}\sum\limits_{i=1}^{N_{k}}A_{k}(\bm{z}_{k}^{a}(t_{i}))}{\sum\limits_{k=0}^{k_{0}}\frac{1}{N_{k}}\sum\limits_{i=1}^{N_{k}}B_{k}(\bm{z}_{k}^{b}(t_{i}))}.
\end{equation}
We define $N_{T}:=\sum\limits_{k=0}^{k_{0}}N_{k}$ to be the total computation times, and thus we actually need to sample $2N_{T}$ times (sample $N_{T}$ times for numerator and denominator respectively) to approximate $I_{2k_{0}}$. 
\begin{rem}
In this paper, the numerical construction of the trajectories are obtained by the BAOAB method to be described in Subsection \ref{subsec-5.1}.
\end{rem}

\subsection{RM-PIMD method for truncated thermal averages}\label{subsec-5.1}
\hspace{4mm} After giving the general numerical method \eqref{4-5} for the truncated thermal average $I_{2k_{0}}$ without specifying sampling the times $N_{k}\,(0\le k\le k_{0})$, we are yet to distribute the computation power for each $k$. A simple way is to arrange the same sampling times for each sub-estimator $\widehat{\mathrm{E}}_{\tilde{\pi}}A_{k}$ ans $\widehat{\mathrm{E}}_{\tilde{\pi}}B_{k}$, which means $N_{k}=N_{0}$ for all $k$ from $0$ to $k_{0}$ and we name this numerical method the path integral molecular dynamics method with a reference measure (RM-PIMD). And the numerical scheme can be presented as
\begin{equation}\label{4-6}
I_{2k_{0}}\approx \frac{\sum\limits_{k=0}^{k_{0}}\widehat{\mathrm{E}}_{\tilde{\pi}}A_{k}(N_{0})}{\sum\limits_{k=0}^{k_{0}}\widehat{\mathrm{E}}_{\tilde{\pi}}B_{k}(N_{0})}
=\frac{\sum\limits_{k=0}^{k_{0}}\frac{1}{N_{0}}\sum\limits_{i=1}^{N_{0}}A_{k}(\bm{z}_{k}^{a}(t_{i}))}{\sum\limits_{k=0}^{k_{0}}\frac{1}{N_{0}}\sum\limits_{i=1}^{N_{0}}B_{k}(\bm{z}_{k}^{b}(t_{i}))}.   
\end{equation}
RM-PIMD is not an optimal way for the choice of $N_{k}$, and it only serves as a comparison to MLMC-PIMD method to be introduced later.
\par  Recall the definition of the total computation times $N_{T}$, the total computation times of RM-PIMD is $N_{T}=(k_{0}+1)N_{0}$. We give a detailed algorithm of the RM-PIMD method below.
\begin{algorithm}[ht]\label{alg-1}
\caption{RM-PIMD}
\hspace*{0.02in} {\bf Input:}
Total computation times $N_{T}$, time step $\Delta t$ and $k_{0}$\\
\hspace*{0.02in} {\bf Output:} %算法的结果输出
Truncated thermal average $I_{2k_{0}}$
\begin{algorithmic}[1]
\State Compute the sampling times $N_{0}$ for each sub-estimator using $N_{T}=(k_{0}+1)N_{0}$.
\State With the BAOAB method, obtain sub-estimators in \eqref{4-4} to compute $\widehat{\mathrm{E}}_{\tilde{\pi}}A_{k}$ and $\widehat{\mathrm{E}}_{\tilde{\pi}}B_{k}$.
\State \Return Compute $I_{2k_{0}}$ with the numerical approximation \eqref{4-6}.
\end{algorithmic}
\end{algorithm}

\subsection{MLMC-PIMD method for truncated thermal averages}
\hspace{4mm} After showing the decaying property of the variances of the sub-estimators $\widehat{\mathrm{E}}_{\tilde{\pi}}A_{k}$ and $\widehat{\mathrm{E}}_{\tilde{\pi}}B_{k}$ and the increasing sampling difficulty as $k$ increases, we intend to optimize the numbers of samples allocated to each sub-estimator to minimize the total computational cost. We fix a variance and optimize each sub-estimator's sampling number to achieve a minimal computational cost in the spirit of multi-level Monte Carlo method\cite{giles2013multilevel,Giles2008MultilevelMC,giles_2015}. We emphasize that this is the main motivation of constructing the numerical scheme below.
\par Now we elaborate our analysis. Consider the denominator of numerical scheme \eqref{4-5}
\begin{equation*}
    \sum\limits_{k=0}^{k_{0}}\widehat{\mathrm{E}}_{\tilde{\pi}}B_{k}=\sum\limits_{k=0}^{k_{0}}\frac{1}{N_{k}}\sum\limits_{i=1}^{N_{k}}B_{k}(\bm{z}_{k}^{b}(t_{i})).
\end{equation*}
For different $\widehat{\mathrm{E}}_{\tilde{\pi}}B_{k}\, (0\leq k \leq k_{0})$, we sample $N_{k}$ times.
According to the expression for variance
\begin{equation*}
    \operatorname{Var}(X)=\mathrm{E}(X^{2})-(\mathrm{E}X)^{2}\leq \mathrm{E}(X^{2})\leq \max (|X|)^{2},
\end{equation*}
for simplicity we assume that $z^{b}_{k}(t_{i})$ are independent, then we have
\begin{align}\label{4-7}
    \operatorname{Var}(\sum\limits_{k=0}^{k_{0}}\widehat{\mathrm{E}}_{\tilde{\pi}}B_{k})
    &=\operatorname{Var}(\sum\limits_{k=0}^{k_{0}}\frac{1}{N_{k}}\sum\limits_{i=1}^{N_{k}}B_{k}(z^{b}_{k}(t_{i})))\notag=\sum\limits_{k=0}^{k_{0}}\frac{1}{N_{k}^{2}}\sum\limits_{i=1}^{N_{k}}\operatorname{Var}(B_{k}(\bm{z}_{k}^{b}(t_{i})))\notag\\
    &\le \sum\limits_{k=0}^{k_{0}}\frac{1}{N_{k}^{2}}\sum\limits_{i=1}^{N_{k}}\max(|B_{k}(\bm{z}_{k}^{b}(t_{i}))|)= \sum\limits_{k=0}^{k_{0}}\frac{1}{N_{k}}\max(|B_{k}(\bm{z})|)^{2}.
\end{align}
For $|\bm{\ell}|=2k$, we assume the $2k$ kinks occur after the surface index $\ell_{i_{1}},\dots,\ell_{i_{2k}}(1\leq i_{1}<\dots<i_{2k}\leq N)$, then we have 
\begin{align}\label{4-8}
    \frac{e^{-\beta_{N}H_{N}(\bm{q},\bm{p},\bm{\ell})}}{e^{-\beta_{N}H_{N}(\bm{q},\bm{p},\bm{\ell}_{0})}}&
    =\frac{e^{-\beta_{N}\sum\limits_{k=1}^{N}V(q_{k},\ell_{k},\ell_{k+1})}\prod\limits\limits_{j=1}^{2k}\sinh(\beta_{N}|V_{01}(q_{i_{j}})|)\prod\limits_{j\neq i_{1},\dots,i_{2k}}\cosh(\beta_{N}|V_{01}(q_{j})|)}{e^{-\beta_{N}\sum\limits_{k=1}^{N}V_{00}(q_{k})}\prod\limits_{j=1}^{N}\cosh(\beta_{N}|V_{01}(q_{j})|)} \notag\\
    &=e^{-\beta_{N}\sum\limits_{k=1}^{N}(V(q_{k},\ell_{k},\ell_{k+1})-V_{00}(q_{k}))}\prod\limits_{j=1}^{2k}\frac{\sinh(\beta_{N}|V_{01}(q_{i_{j}})|)}{\cosh(\beta_{N}|V_{01}(q_{i_{j}})|)}.
\end{align}
We further assume $V_{00}(q)-V_{11}(q)\leq C_{1}$ and $|V_{01}(q)|\leq C_{2}$ as what we assumed in Theorem \ref{thm-1}. According to \eqref{p-4}, we have
\begin{equation}\label{4-9}
    \frac{\sinh(\beta_{N}|V_{01}(q_{i_{j}})|)}{\cosh(\beta_{N}|V_{01}(q_{i_{j}})|)}\le MC_{2}\beta_{N}.
\end{equation}
Using \eqref{p-10}, we get
\begin{equation}\label{4-10}
    e^{-\beta_{N}\sum\limits_{k=1}^{N}(V(q_{k},\ell_{k},\ell_{k+1})-V_{00}(q_{k}))}\le e^{C_{1}}.
\end{equation}
According \eqref{4-9} and \eqref{4-10}, we can bound \eqref{4-8}
\begin{equation}\label{4-11}
\frac{e^{-\beta_{N}H_{N}(\bm{q},\bm{p},\bm{\ell})}}{e^{-\beta_{N}H_{N}(\bm{q},\bm{p},\bm{\ell}_{0})}}
\le e^{C_{1}} (\frac{MC_{2}}{N})^{2k}=e^{C_{1}}(\frac{C}{N})^{2k}.
\end{equation}
Then to bound $B_{k}(\bm{z})$, we notice $\{\bm{\ell}:|\bm{\ell}|=2k\}$ contains $2\tbinom{N}{2k}$ different index sequences,
\begin{equation}\label{4-12}
    \max(|B_{k}(\bm{z})|)
    \leq 2\binom{N}{2k}e^{C_{1}}(\frac{C}{N})^{2k}
    \leq \frac{2 C^{N} e^{C_{1}}}{(2k)!},
\end{equation}
where we use $N^{-2k}\tbinom{N}{2k}\leq \frac{1}{(2k)!}$.
Use the inequality \eqref{4-12} in the estimation of variance \eqref{4-7}, we have
\begin{equation}\label{4-13}
    \operatorname{Var}(\sum\limits_{k=0}^{k_{0}}\widehat{\mathrm{E}}_{\tilde{\pi}}B_{k})
    \leq \sum\limits_{k=0}^{k_{0}}\frac{4C^{2N}e^{2C_{1}}}{N_{k}((2k)!)^{2}}.
\end{equation}
As for the estimation of the total computational cost, for different $k$, according to the form of $B_{k}(\bm{z})$, we need to sum $2\binom{N}{2k}$ times to generate $B_{k}(\bm{z}_{k}^{b}(t_{i}))$. Thus the total computational cost of the denominator is
\begin{equation*}
    \text{Total computational cost}\propto\sum\limits_{k=0}^{k_{0}} 2N_{k}\binom{N}{2k}.
\end{equation*}
We notice as $k$ increases from $0$ to $\frac{1}{2}\lfloor\frac{N}{2}\rfloor$, the variances of sub-estimators $\widehat{\mathrm{E}}_{\tilde{\pi}}B_{k}$ decrease exponentially while the computational cost increases. Thus the total variance is mainly contributed by the variances of sub-estimators $\widehat{\mathrm{E}}_{\tilde{\pi}}B_{k}$ with relatively small $k$ while the total computational cost is mainly dominated by the sub-estimators $\widehat{\mathrm{E}}_{\tilde{\pi}}B_{k}$ with relatively large $k$. 
\par Following the spirit of multi-level Monte Carlo method \cite{Giles2008MultilevelMC,giles_2015,giles2013multilevel}, we sample more times on the sub-estimators $\widehat{\mathrm{E}}_{\tilde{\pi}}B_{k}$ with relatively small $k$ to reduce the total variance while saving the average computational cost. Quantitatively, for a prescribed small $\epsilon$, we optimize the choice of $N_{k}\, (0\leq k\leq k_{0})$ to reduce the computational cost subjective to the condition that the total variance is less than $\epsilon$. And finding the optimal choice of $N_{k}$ connects to a conditional extremum problem:
\begin{equation*}
    \text{minimize}\, :\sum\limits_{k=0}^{k_{0}}2N_{k}\binom{N}{2k},
    \quad\text{while} \,  :\sum\limits_{k=0}^{k_{0}}\frac{4C_{2}^{2N}e^{2|C_{3}|}}{N_{k}((2k)!)^{2}}\leq \epsilon.
\end{equation*}
Using the Lagrangian multiplier method, it is easy to conclude that (Readers can see the detailed derivation of this result in Appendix A):
\begin{equation}\label{4-14}
    N_{k}=\frac{4C^{2N}e^{2|C_{3}|}}{\epsilon}\left(\sum\limits_{k=0}^{k=k_{0}}\frac{1}{(2k)!}\sqrt{\tbinom{N}{2k}}\right)\left(\tbinom{N}{2k}\right)^{-\frac{1}{2}}\frac{1}{(2k)!}.
\end{equation}
and the computational cost achieves to its minima.
\par From the above analysis, the times $N_{k}$ we need to sample for sub-estimators $\widehat{\mathrm{E}}_{\tilde{\pi}}A_{k}$ and $\widehat{\mathrm{E}}_{\tilde{\pi}}B_{k}$ should be proportional to $i_{k}:=\left(\tbinom{N}{2k}\right)^{-\frac{1}{2}}\frac{1}{(2k)!}$. Thus when we use the numerical scheme \eqref{4-5} to compute the truncated thermal average $I_{2k_{0}}$ with a given total computation times $N_{T}$, the sampling times $N_{k}$ for sub-estimators $\widehat{\mathrm{E}}_{\tilde{\pi}}A_{k}$ and $\widehat{\mathrm{E}}_{\tilde{\pi}}B_{k}$ should satisfy
\begin{equation}\label{4-15}
    \sum\limits_{k=0}^{k_{0}}N_{k}=N_{T}
    \quad\text{and}\quad
    N_{0}:N_{1}:\cdots:N_{k_{0}}=i_{0}:i_{1}:\cdots:i_{k_{0}}.
\end{equation}
And we name this numerical method the multi-level Monte Carlo path integral molecular dynamics method (MLMC-PIMD). We give a complete MLMC-PIMD algorithm for the computation of truncated thermal average $I_{2k_{0}}$.
\begin{algorithm}[ht]\label{alg-2}
\caption{MLMC-PIMD}
\hspace*{0.02in} {\bf Input:}
Total computation times $N_{T}$, time step $\Delta t$ and $k_{0}$\\
\hspace*{0.02in} {\bf Output:} %算法的结果输出
Truncated thermal average $I_{2k_{0}}$
\begin{algorithmic}[1]
\State Compute the sampling times $N_{k}\,(0\le k\le k_{0})$ for each sub-estimator satisfying $N_{0}:\cdots:N_{k_{0}}=i_{0}:\cdots:i_{k_{0}}$ and $\sum\limits_{k=0}^{k_{0}}N_{k}=N_{T}$.
\State With BAOAB method, obtain sub-estimators in \eqref{4-4} to sample the trajectory $N_{k}$ times to compute $\widehat{\mathrm{E}}_{\tilde{\pi}}A_{k}$ and $\widehat{\mathrm{E}}_{\tilde{\pi}}B_{k}$ for each $k$.
\State \Return Compute $I_{2k_{0}}$ with the numerical approximation \eqref{4-5}.
\end{algorithmic}
\end{algorithm}

\section{Sub-estimator and experiment report}\label{sec-5}
\subsection{The BAOAB method for sub-estimator in MLMC-PIMD}
\hspace{4mm}
In numerical experiments, we divide the Langevin dynamics constructed in \eqref{4-3} to sample the distribution to compute the integral $\mathrm{E}_{\tilde{\pi}}A_{k}$ and $\mathrm{E}_{\tilde{\pi}}B_{k}$. First we fix an initial point $\bm{z}_{0}=(\bm{q}(0),\bm{p}(0))$ according to the Gaussian distribution $\mathcal{N}(0,M\beta_{N}^{-1})$ at time $t_{0}=0$ and then set the time step $\Delta t$. In this article, we apply the BAOAB method for the sub-estimator of Langevin dynamics, which means we repeat the BAOAB method in each time interval $[t_{n},t_{n}+\Delta t]\, (t_{n}=n\Delta t)$ until we reach time $T\gg \Delta t$.
\par
We give a brief introduction to the BAOAB method for Langevin dynamics in the context of 
Langevin thermostat\cite{leimkuhler2013robust}. The Langevin dynamics can be written as $$\left\{\begin{array}{l}
{\mathrm{d} \bm{q}=M^{-1} \bm{p} \mathrm{d} t}, \\
{\mathrm{d} \bm{p}=-\nabla_{\bm{q}} H_{N} \mathrm{d} t-\gamma \bm{p} \mathrm{d} t+\sqrt{2 \gamma \beta_{N}^{-1}M}  \mathrm{d} \bm{W}.}
\end{array}\right.$$In the BAOAB method, the Langevin dynamics is divided into three parts, the kinetic part (part $\text{``A''}$):$$\left\{\begin{array}{l}
{\mathrm{d} \bm{q}=M^{-1} \bm{p} \mathrm{d} t}, \\
{\mathrm{d} \bm{p}=0,}
\end{array}\right.$$
the potential part (part $\text{``B''}$):
$$\left\{\begin{array}{l}
{\mathrm{d} \bm{q}=0}, \\
{\mathrm{d} \bm{p}=-\nabla_{\bm{q}} H_{N} \mathrm{d} t,}
\end{array}\right.$$
and the Langevin thermostat part (part $\text{``O''}$):
$$\left\{\begin{array}{l}
{\mathrm{d} \bm{q}=0}, \\
{\mathrm{d} \bm{p}=-\gamma \bm{p} \mathrm{d} t+\sqrt{2 \gamma \beta_{N}^{-1}M}  \mathrm{d} \bm{W}.}
\end{array}\right.$$
An advantage of this method is that each of these splitted parts can be integrated individually. For example, after we know $\bm{z}_{0}=(\bm{q}(0),\bm{p}(0))$ the Langevin thermostat part (part $\text{``O''}$) has a solution:
$$\left\{\begin{array}{l}
{\bm{q}(t)=\bm{q}(0)}, \\
{\bm{p}(t)=e^{-\gamma t} \bm{p}(0)+\sqrt{\left(1-e^{-2 \gamma t}\right)\left(\beta_{N}^{-1} M\right)} \bm{W}(t),}
\end{array}\right.$$
where $\bm{W}$ denotes a multi-dimensional Gaussian random variable.
\par By using BAOAB method, first we solve part B and part A in order within the time step $\Delta t/2$, then solve part O in a full time step $\Delta t$ and finally solve part A and part B within the time step $\Delta t/2$. In integrating each small step, we use their exact solutions. It is shown \cite{leimkuhler2013robust} that the BAOAB method has a higher accuracy than other splitting methods and can be implemented with larger time steps.

\subsection{Numerical results}
\hspace{4mm}
To test the validity of MLMC-PIMD method, we need to test the convergence of the sub-estimators $\widehat{\mathrm{E}}_{\tilde{\pi}}A_{k}$ and $\widehat{\mathrm{E}}_{\tilde{\pi}}B_{k}$, the convergence of the truncated thermal average $I_{2k_{0}}$ with the increasing of $k_{0}$ and the improved numerical performance of MLMC-PIMD method in terms of simulation time and accuracy compared to RM-PIMD method and PIMD-SH method. We implement numerical experiments on the following example.

\subsubsection{Test example}
\hspace{4mm}
The potentials are chosen to be one-dimensional taking the form:
\begin{equation}\label{5-1}
\left\{\begin{array}{l}
{V_{00}=x^{2}+2(1-\cos(x))-3e^{-(x-1)^{2}}-2e^{-(x-1.5)^2}+3}, \\
{V_{11}=x^{2}+4(1-\cos(x))-2e^{-(x-1)^{2}}+3}, \\
{V_{01}=V_{10}=e^{-x^{2}},}
\end{array}\right.
\end{equation}
where we choose $V_{00} \le V_{11}$ to satisfy the assumption in Theorem \ref{thm-1}. The two energy surfaces respectively achieve their minimas around $x=0.3$ and $x=0.8$ and almost intersect around $x=-0.5$, while the off-diagonal potential is symmetric and achieves its maxima at $x=0$. Thus the potentials are asymmetric in this example and the location where the equilibrium distribution is mainly concentrated deviates from the most active hopping area, which makes this test example more numerically challenging. We plot the potentials on position interval $(-\pi,\pi)$ in the left picture of Figure 1.
\begin{figure}[ht] \label{fig-1}
\centering 
\begin{minipage}[ht]{0.48\textwidth} 
	\centering 
	\includegraphics[width=7cm]{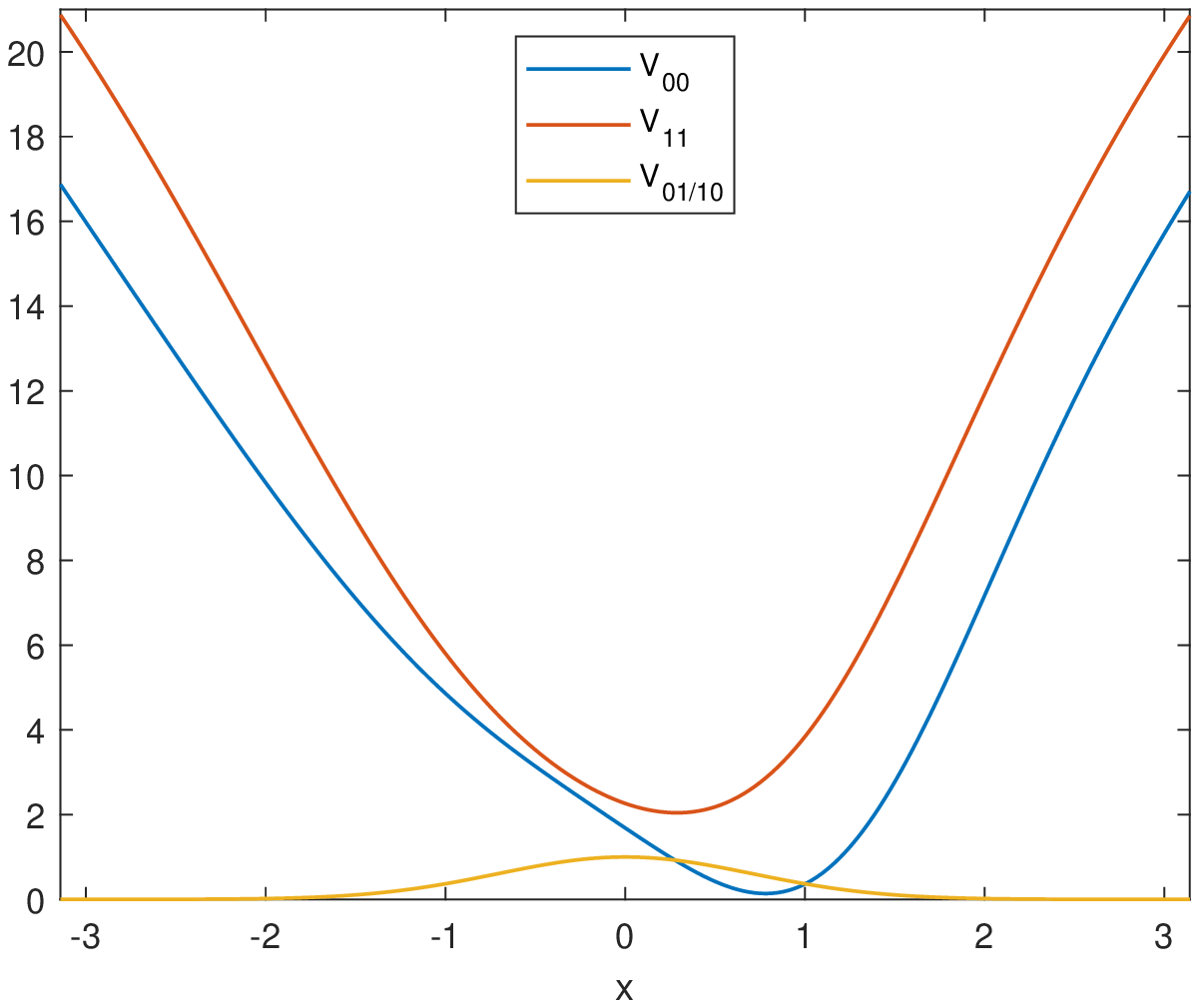}
\end{minipage} 
\begin{minipage}[ht]{0.48\textwidth} 
	\centering 
	\includegraphics[width=7cm]{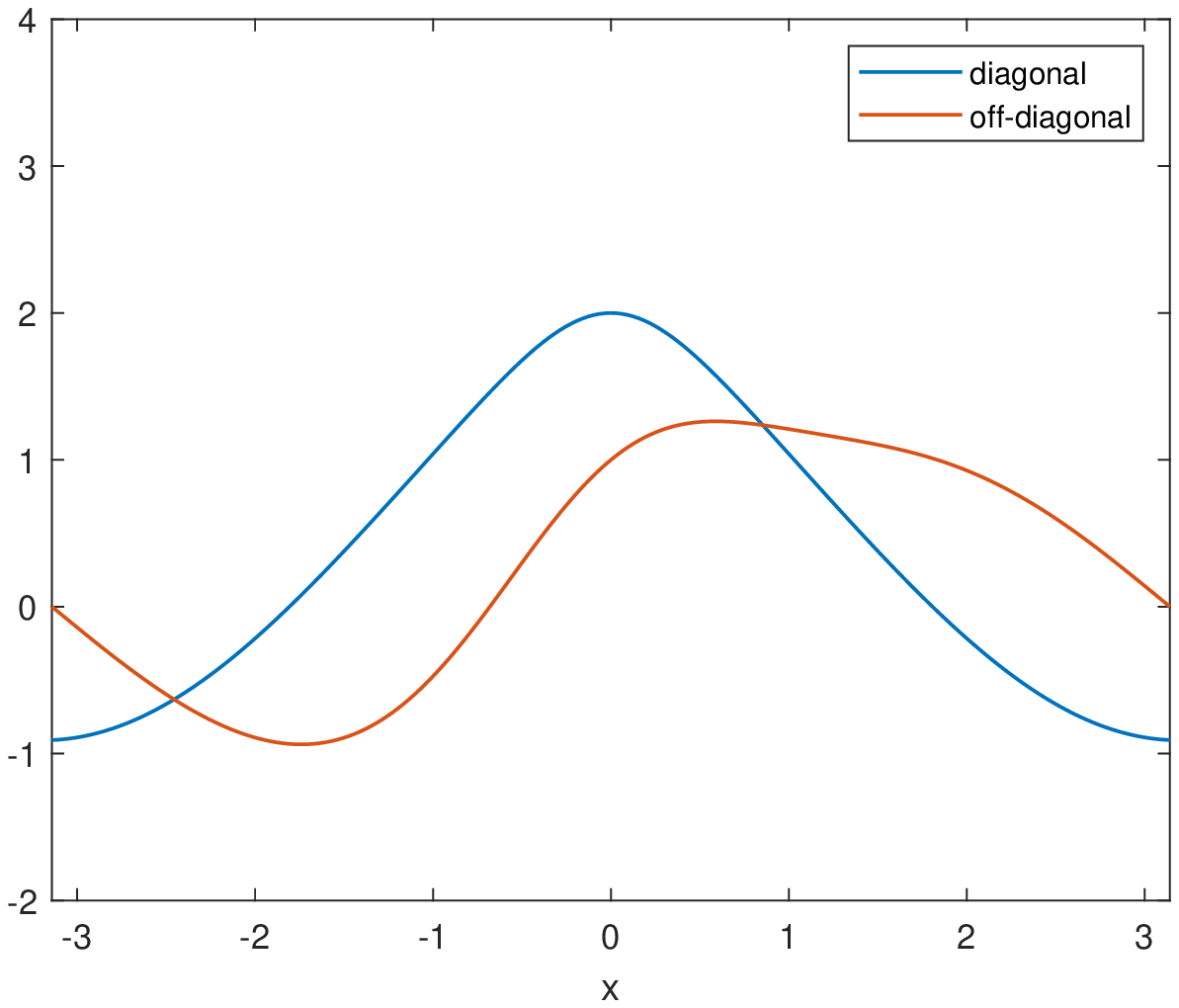}
\end{minipage} 
\caption{Left: Potentials defined in \eqref{5-1}. Right: Observable defined in \eqref{5-2}}
\end{figure}
\par
We then set other parameters for the rest of this article: $\beta=1,M=1,N=16$ and $\gamma=1$.
We let the observable takes the form (the picture is shown in the right part of Figure 1):
\begin{equation}\label{5-2}
A=\left[\begin{array}{cc}
{\frac{1}{1+x^{2}}+\cos(x)} & {e^{-x^{2}}+\sin(x)} \\
{e^{-x^{2}}+\sin(x)} & {\frac{1}{1+x^{2}}+\cos(x)}
\end{array}\right],
\end{equation}
where we choose the off-diagonal components of $A$ to be non-zero since in most applications we cannot assume $A$ to be diagonal. And the asymmetric diagonal components of $A$ make the computation more challenging.

Using pseudo-spectral approximation method, we obtain the reference value for the thermal average $\langle\hat{A\rangle}$ is $0.987553$ with all parameters we set above.

\subsubsection{Convergence of sub-estimators in MLMC-PIMD}
\hspace{4mm}
To test the convergence of BAOAB method in sub-estimators, we take the computation of $\widehat{\mathrm{E}}_{\tilde{\pi}}A_{k}$ for different $k$ as an example after we fix $N_{k}=2\times 10^{5}$, set $\beta=1,M=1,N=16,\Delta t=0.5\times 10^{-2}$ and $\gamma=1$ and use the potentials in \eqref{5-1} and observable in \eqref{5-2}. 
\par When $k=0$, we get $\left\{\bm{\ell}: \left| \bm{\ell} \right| = 0\right\}=\left\{(0,\cdots ,0),(1,\cdots,1)\right\}$, which means the integrated function $A_{0}$ is a sum of $2$ different parts. In general, as $k$ grows larger from $0$ to $\frac{1}{2}\lfloor \frac{N}{2} \rfloor$, $\left\{\bm{\ell}: \left| \bm{\ell} \right| = 2k\right\}$ contains more elements, which means we need to sum more times to get $A_{k}$. The following pictures (Figure 2-3) show different paths of integral on the time interval $[0,1000]$ for different $k$ from $0$ to $3$ to obtain $\widehat{\mathrm{E}}_{\tilde{\pi}}A_{k}$, from which we can easily see the convergence of sub-estimators using BAOAB method compared to the reference value marked by the straight line.
\begin{figure}[ht]\label{fig-2}
\centering 
\includegraphics[width=11cm]{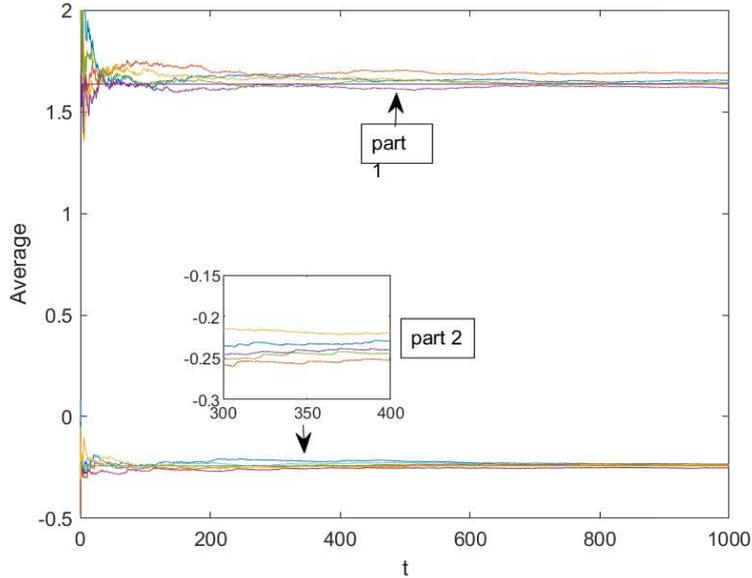}
\caption{Sub-estimators $\widehat{\mathrm{E}}_{\tilde{\pi}}A_{k}$ in \eqref{4-4}. Part 1: Different paths of sub-estimator $\widehat{\mathrm{E}}_{\tilde{\pi}}A_{0}$. Part 2: Different paths of sub-estimator $\widehat{\mathrm{E}}_{\tilde{\pi}}A_{1}$. Parameters for both parts: $\beta=1,M=1,N=16,\Delta t=0.5\times 10^{-2}$ and $\gamma=1$.}
\end{figure}
\begin{figure}[ht]\label{fig-3}
\centering 
\includegraphics[width=11cm]{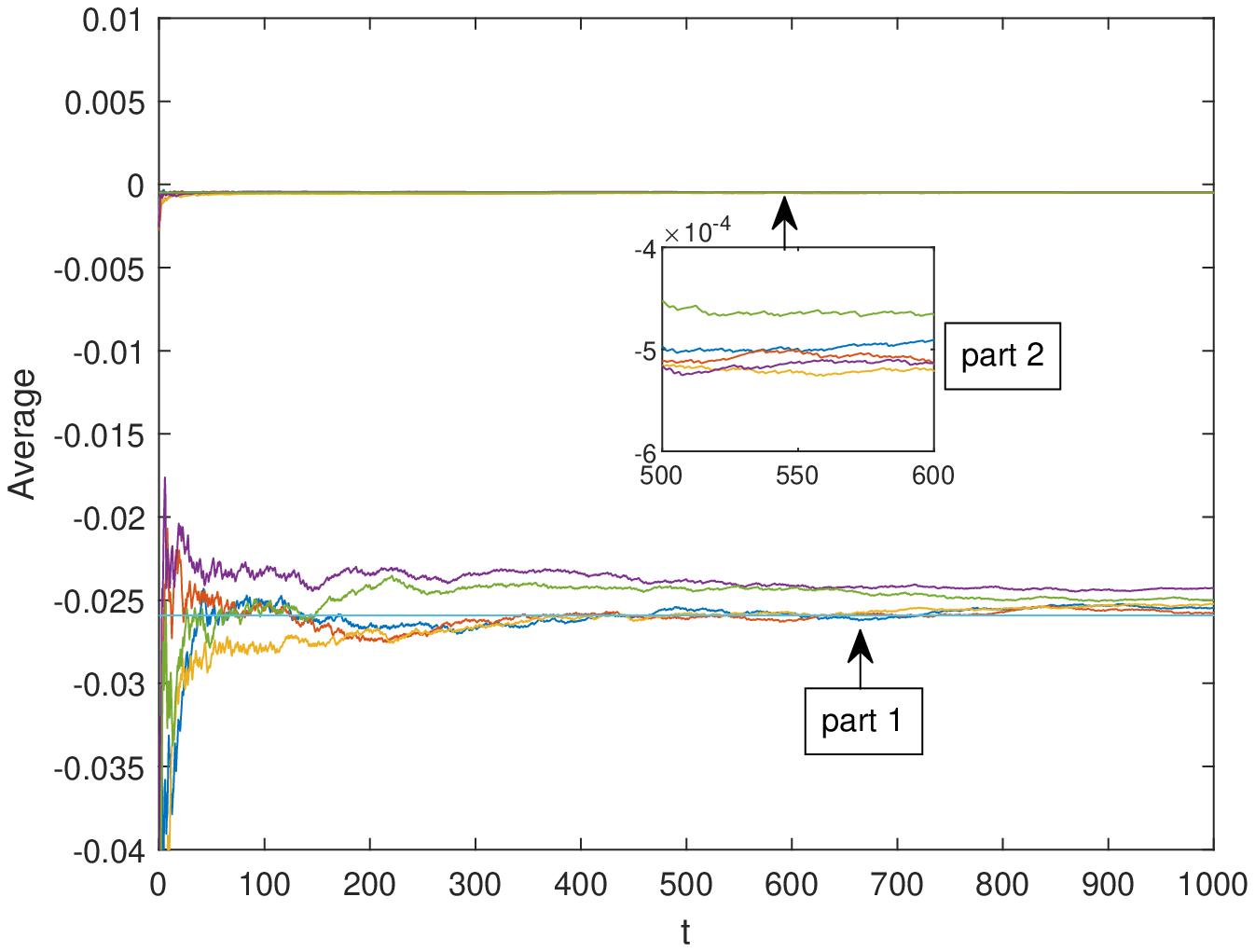}
\caption{Sub-estimators $\widehat{\mathrm{E}}_{\tilde{\pi}}A_{k}$ in \eqref{4-4}. Part 1: Different paths of sub-estimator $\widehat{\mathrm{E}}_{\tilde{\pi}}A_{2}$. Part 2: Different paths of sub-estimator $\widehat{\mathrm{E}}_{\tilde{\pi}}A_{3}$. Parameters for both parts: $\beta=1,M=1,N=16,\Delta t=0.5\times 10^{-2}$ and $\gamma=1$.}
\end{figure}
\par 
Moreover, we can see the absolute values of averages $|\widehat{\mathrm{E}}_{\tilde{\pi}}A_{k}|$ (or $|\mathrm{E}_{\tilde{\pi}}A_{k}|$) decrease fast with the increasing of $k$, which demonstrates $A_{k}$ decrease fast while $k$ grows. Because of the fast decrease of $A_{k}$, we can also show the variance of sub-estimator decreases with the increasing of 
$k$ as what has been shown in the Table 1, where we list different variances of sub-estimators $\widehat{\mathrm{E}}_{\tilde{\pi}}A_{k}\,(0\le k\le 3)$ and see the variances decrease as $k$ increases.
\begin{table}[ht]
	\begin{center}  
		\begin{tabular}{|l|l|l|l|l|}  
			
			\hline  
			
			Sub-estimator & $\widehat{\mathrm{E}}_{\tilde{\pi}}A_{0}$ & $\widehat{\mathrm{E}}_{\tilde{\pi}}A_{1}$ & $\widehat{\mathrm{E}}_{\tilde{\pi}}A_{2}$ &
			$\widehat{\mathrm{E}}_{\tilde{\pi}}A_{3}$ \\ \hline
			Variance$(\times 10^{-4})$ &7.551 & $5.541\times 10^{-1}$ & $3.150\times 10^{-3}$ & $3.051\times 10^{-6}$
			\\ \hline
			\end{tabular}
		\caption{Variances for different sub-estimators $\widehat{\mathrm{E}}_{\tilde{\pi}}A_{k}\,(0\le k\le 3)$. Parameters: $\beta=1,M=1,N=16,\Delta t=0.5\times 10^{-2}$ and $\gamma=1$.}
	\end{center}  
\end{table}

\subsubsection{Convergence of truncated thermal averages}
\hspace{4mm}
To verify the approximation property of the truncated thermal averages, we need to compute $I_{2k_{0}}$ according to different $k_{0}$. We use RM-PIMD method here to compute $I_{2k_{0}}$, which means we sample the same $N_{0}$ times for each sub-estimator $\widehat{\mathrm{E}}_{\tilde{\pi}}A_{k}$ and $\widehat{\mathrm{E}}_{\tilde{\pi}}B_{k}$, and then see how the MSE of these different numerical outcomes, which are actually random variables, change according to $k_{0}$ and $N_{0}$. To compute the MSE of RM-PIMD method with respect to the $I_{2k_{0}}$ needed to estimate, we use the following unbiased estimation
\begin{equation*}
MSE\approx\frac{1}{N}\sum\limits_{k=1}^{N}(X_{k}-I_{2k_{0}})^{2},
\end{equation*}
where $X_{k}\,(1\le k\le N)$ are independent outcomes obtained from the RM-PIMD (Algorithm \ref{alg-1}) to compute $I_{2k_{0}}$. The following figure (Figure 4) shows the convergence of truncated thermal averages $I_{2k_{0}}$.
\begin{figure}[ht] 
	\centering 
	\includegraphics[width=11cm]{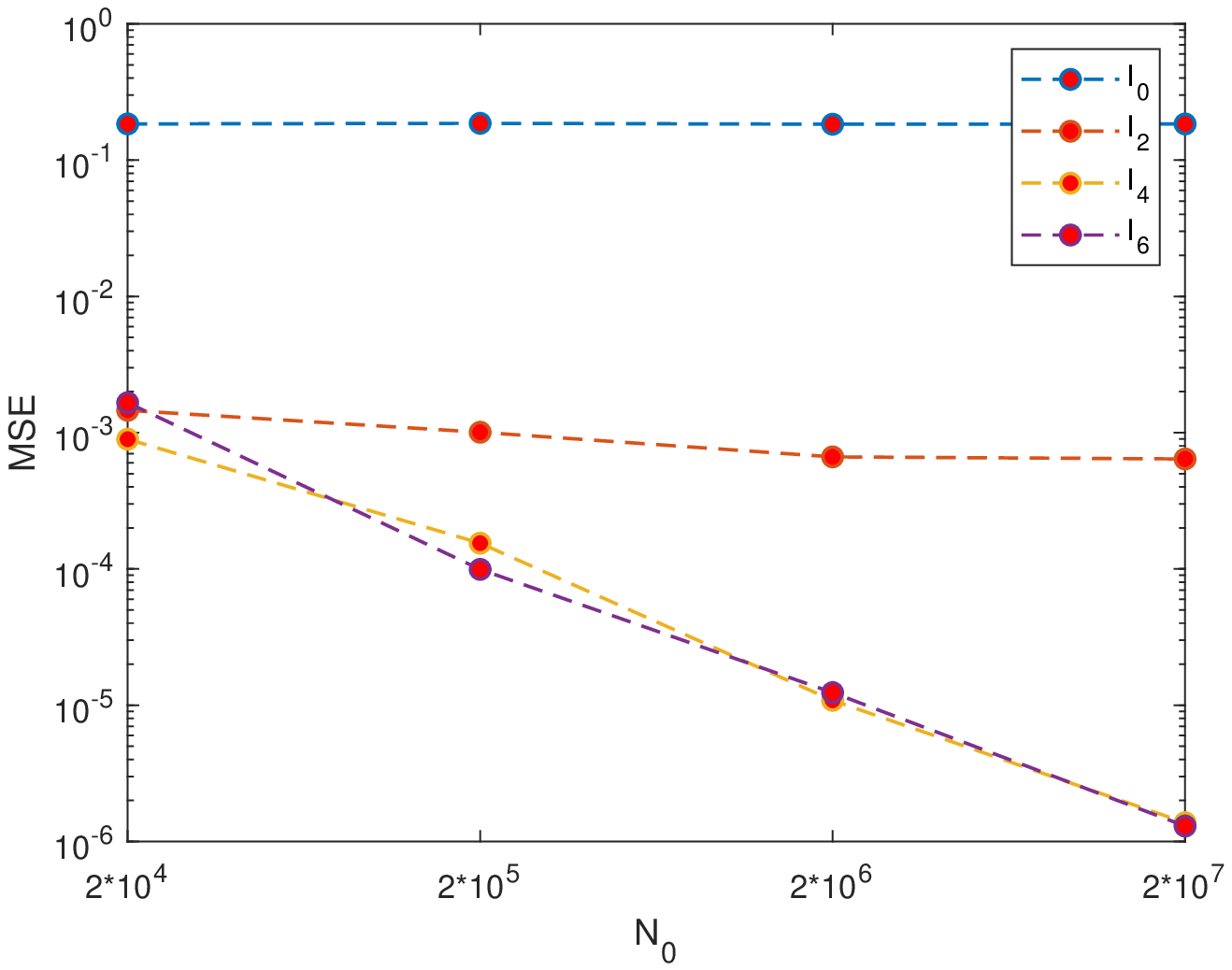} 
	\caption{MSE of $I_{2k_{0}}$ with different $k_{0}$ and $N_{0}$. Parameters: $\beta=1,M=1,N=16,\Delta t=0.5\times 10^{-2}$ and $\gamma=1$. $N_{0}$: Sampling times for each sub-estimators $\widehat{\operatorname{E}}_{\tilde{\pi}}A_{k}$ and $\widehat{\operatorname{E}}_{\tilde{\pi}}B_{k}$.} 
\end{figure}
\par 
As the Theorem \ref{thm-1} shows, the biases of the $I_{2k_{0}}$ become smaller when $k_{0}$ becomes larger. We notice from Figure 4 that our estimation of $I_{0}\,(k_{0}=0)$ and $I_{2}\,(k_{0}=1)$ have larger and almost constant MSE compared to $I_{4}\,(k_{0}=2)$ and $I_{6}\,(k_{0}=3)$, which is because the truncated thermal averages $I_{0}$ and $I_{2}$ have larger biases and the larger biases eliminate the effect of the decrease of variance when $N_{0}$ gets larger. When $k_{0}=2 \,\text{or}\, 3$, we can see the biases of $I_{2k_{0}}$ become nearly negligible and the MSE of the estimations is mainly influenced by its variance, which becomes smaller as the sampling times $N_{0}$ of each sub-estimator becomes larger. So in our case, when we consider the truncated thermal averages $I_{4}$ and $I_{6}$, the estimations have approximation property to the thermal average $\langle\widehat{A}\rangle$ with small enough MSE when we choose a large $N_{0}$.

\subsubsection{Comparison of MLMC-PIMD against RM-PIMD}
\hspace{4mm}
In this part, we test the validity of MLMC-PIMD method we derived in Section \ref{sec-4}. We choose to compute $I_{10}\,(k_{0}=5)$. From the numerical results in 5.2.3, we can see $I_{4}$ serves as a good approximation to the thermal average $\langle\widehat{A}\rangle$ and so does $I_{10}$. According to the analysis in Section \ref{sec-4}, for sub-estimators $\widehat{\mathrm{E}}_{\tilde{\pi}}A_{k}$ and $\widehat{\mathrm{E}}_{\tilde{\pi}}B_{k}$, the sampling times $N_{k}$ should be proportional to $\left(\tbinom{16}{2k}\right)^{-\frac{1}{2}}\frac{1}{(2k)!}$. We use MLMC-PIMD to compute $I_{10}$ when we set six different total computation times $N_{T}=2n\times 10^{5}$ with $n$ from $1$ to $6$. To show the comparison of MLMC-PIMD against RM-PIMD, we use the RM-PIMD method to compute $I_{10}$ under the condition $N_{T}=12\times10^{5}$ and record the MSE's and simulation time of these different methods. The numerical outcomes are recorded in the following Table 2.
\begin{table}[ht]
	\begin{center}  
		\begin{tabular}{|l|l|l|l|l|l|l|l|}  
			
			\hline  
			
			$N_{T}$  & $2\times10^{5}$ & $4\times10^{5}$&$6\times10^{5}$&$8\times10^{5}$&$10\times10^{5}$&$12\times10^{5}$&$12\times10^{5}$ \\ \hline
			Numerical Method & MLMC&MLMC&MLMC&MLMC&MLMC&MLMC&RM \\ \hline
			MSE($\times 10^{-3}$)&0.5484 &0.3858  &0.2228&0.1763&0.1155& 0.0677&0.1648 \\ \hline
			Simulation Time(s) &26.50 &44.27 &60.62&75.05&89.93&104.34&105.35 \\ \hline
			
			\end{tabular}
		\caption{MSE and Simulation Time for different Numerical Methods and $N_{T}$. Parameters: $\beta=1,M=1,N=16,\Delta t=0.5\times 10^{-2}$ and $\gamma=1$. MLMC: MLMC-PIMD, RM: RM-PIMD, $N_{T}$: Total Computation Times.}
	\end{center}  
\end{table}
\par 
From Table 2, when we use MLMC-PIMD method with total computation times $N_{T}=8\times 10^{5}$, the MSE of the estimation is almost the same as that using RM-PIMD method but with a larger $N_{T}=12\times 10^{5}$. Compared to RM-PIMD method in our case, MLMC-PIMD can help us save one third of the simulation time to get the same MSE, showing its validity. To show the result more intuitively, we provide the Figure 5 obtained according to Table 2.
\begin{figure}[ht] 
	\centering 
	\includegraphics[width=11cm]{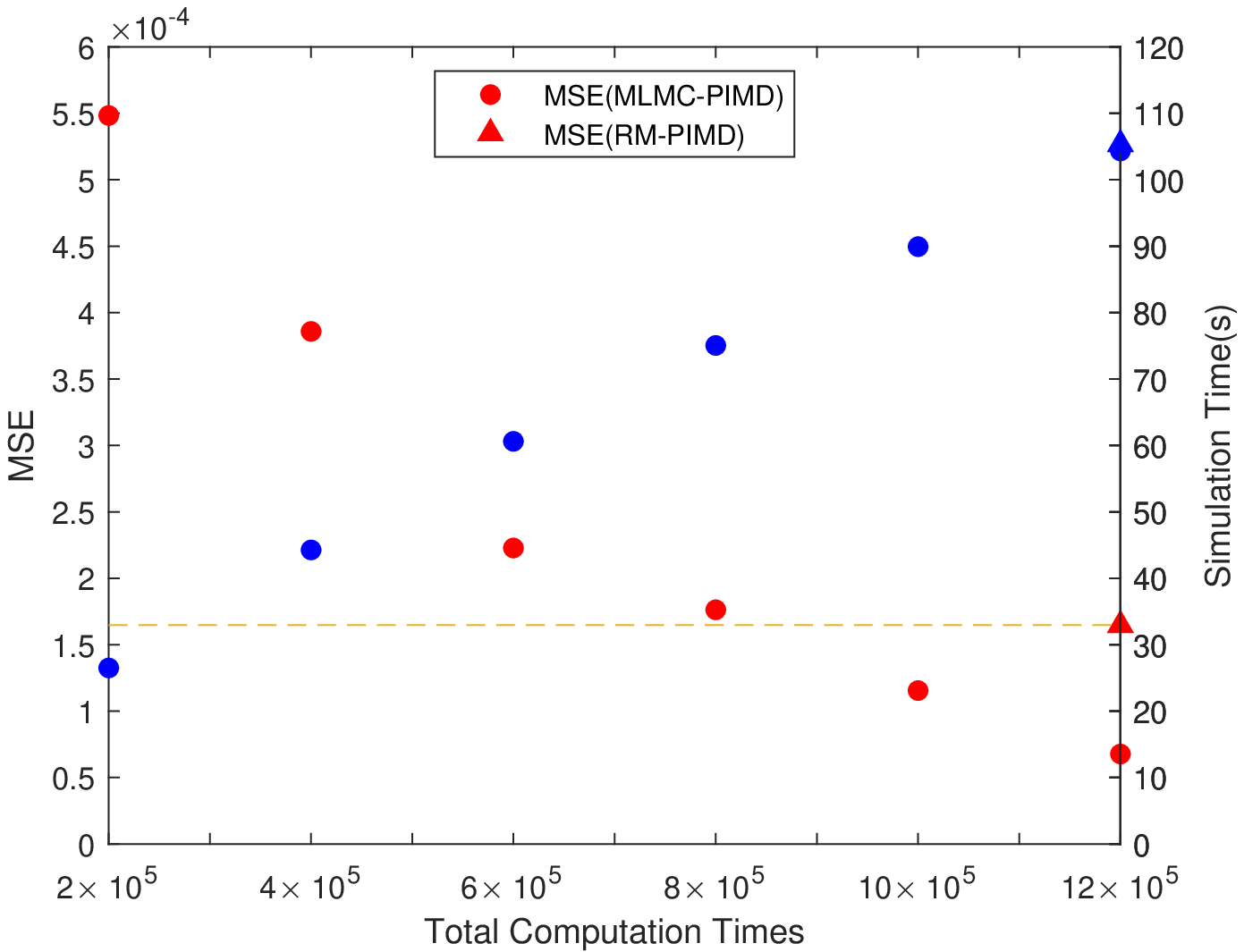}
	\caption{MSE and Simulation Time for different Numerical Methods and $N_{T}$. Parameters: $\beta=1,M=1,N=16,\Delta t=0.5\times 10^{-2}$ and $\gamma=1$. Red points for MSE, blue points for Simulation Time. Circles for MLMC-PIMD method, triangles for RM-PIMD method. Notice the simulation time is almost the same when $N_{T}=12\times 10^{5}$ for both methods.} 
\end{figure}

\subsubsection{Comparison of MLMC-PIMD against PIMD-SH}
\hspace{4mm}In the last part of numerical results, we show the better numerical performance of MLMC-PIMD method compared with PIMD-SH method in \cite{lu2017path}. To show its validity, we intend to compare the MSE of MLMC-PIMD method and PIMD-SH method. Because the time of each iteration of these two methods are different, the total time spent in one simulation of them are different when the total computation times are fixed. Thus comparing the MSE of them when the total computation times are fixed is not appropriate, because PIMD-SH method may get a smaller MSE but spend a much larger amount of time to complete one simulation, which is not enough to show the better numerical performance of MLMC-PIMD method. As a result, we switch to fix the simulation time and compare the MSE, we choose different total computation times $N_{T}$ in MLMC-PIMD method and different time $T$ in PIMD-SH method to make the simulation time of these two methods equal and compare their MSE. The numerical result is shown in Figure 6, from which we can see the MSE of MLMC-PIMD method is uniformly smaller than that of PIMD-SH method when the simulation time changes, showing the validity of MLMC-PIMD method.
\begin{figure}[ht] 
	\centering 
	\includegraphics[width=11cm]{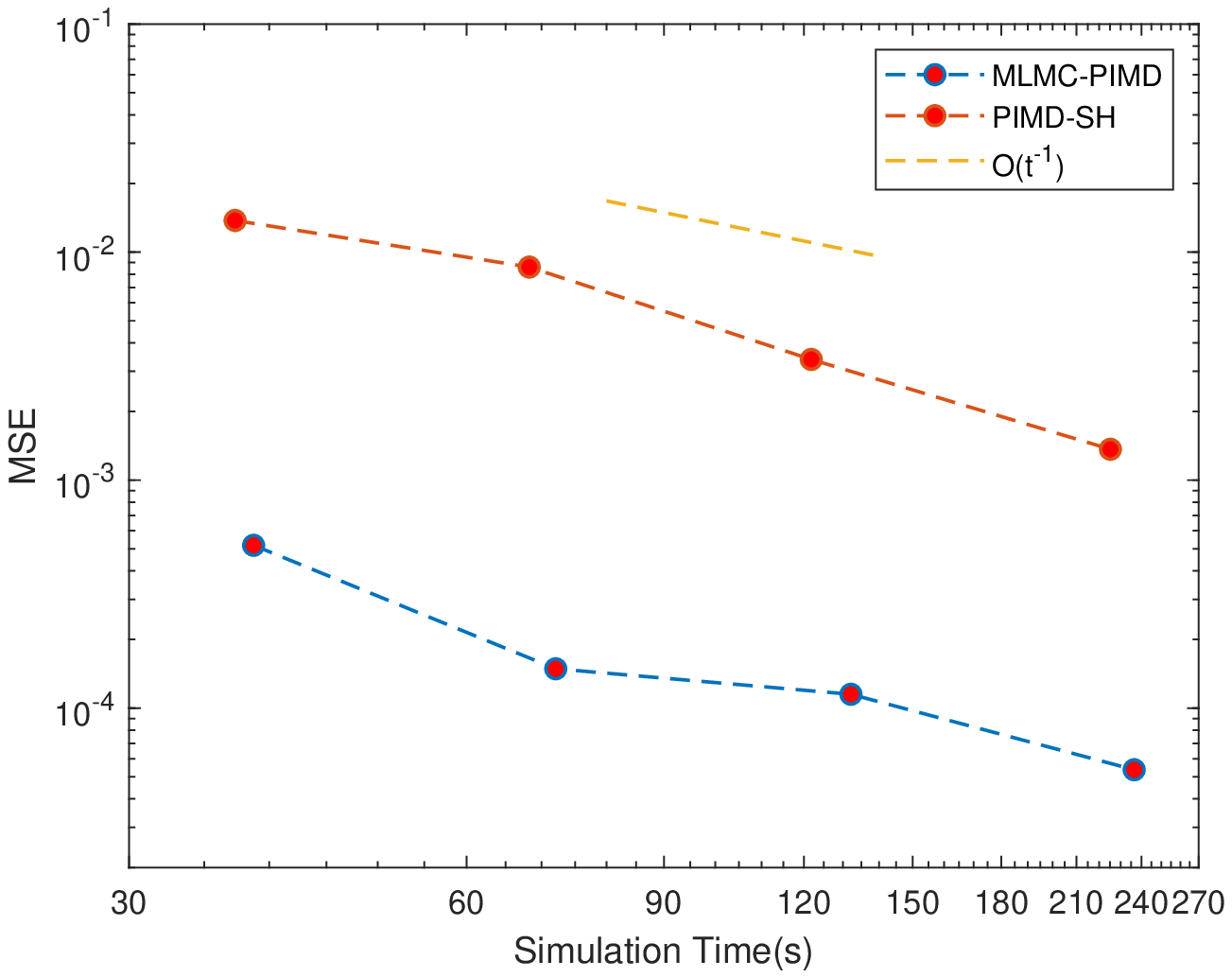} 
	\caption{MSE for different Numerical Methods and Simulation Time. Parameters: $\beta=1,M=1,N=16,\Delta t=0.5\times 10^{-2}$ and $\gamma=1$.} 
\end{figure}
\par To give an intuitive and qualitative interpretation of the better performance of MLMC-PIMD method, we consider and compare one trajectory of sampling in PIMD-SH method $W_{N}[A](\tilde{\bm{z}}(t_{i}))$ and three different sampling trajectories $A_{k}(z(t_{i}))\, (k=0,1,2)$ in sub-estimators $\widehat{\mathrm{E}}_{\tilde{\pi}}A_{k}\,(k=0,1,2)$. We plot the trajectories of $10^{4}$ consecutive samplings after $10^{5}$ equilibrium steps in Figure 7, where the influence of the initial values is weakened.
\begin{figure}[ht] 
	\centering 
	\includegraphics[width=15cm]{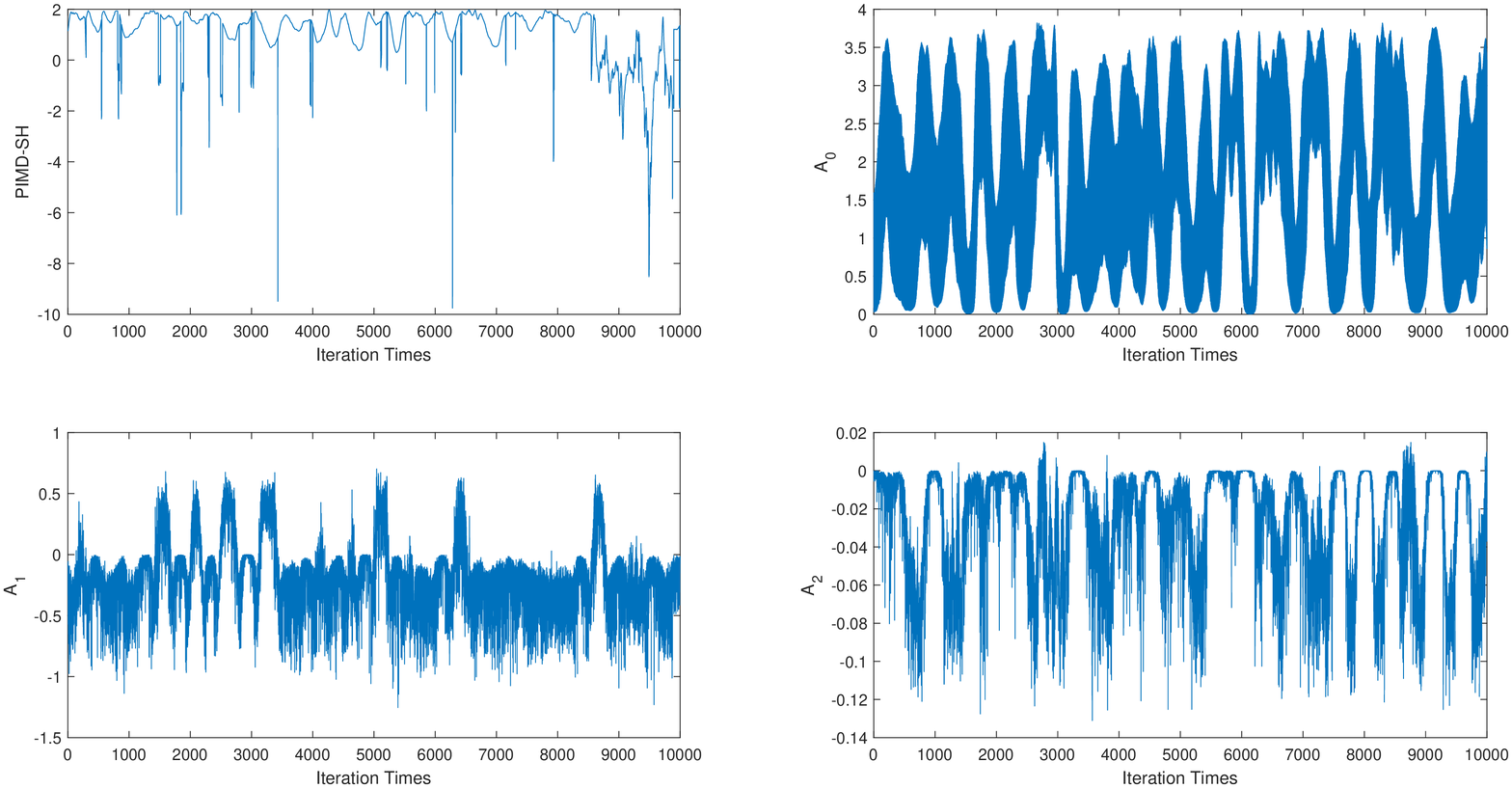} 
	\caption{Trajectories of $10^{4}$ samplings for different estimators. Top left: $W_{N}[A](\tilde{\bm{z}}(t_{i}))$. Top right: Sub-estimator $A_{0}(z(t_{i}))$. Bottom left: Sub-estimator $A_{1}(z(t_{i}))$. Bottom right: Sub-estimator $A_{2}(z(t_{i}))$. Parameters: $\beta=1,M=1,N=16,\Delta t=0.5\times 10^{-2}$ and $\gamma=1$.}
\end{figure}
\par From the picture of the trajectory of PIMD-SH method (top left picture of Figure 7), we notice the trajectory visits the extended configurations with a large kink number in a small probability and thus such visits are rare events. As a result, it takes us a longer time to get enough samplings on these configurations, which affects the numerical performance over the whole trajectory.

\section{Conclusion and further study}\label{sec-6}
\hspace{4mm}In our work, we consider the truncation of the extended ring polymer representation for the thermal average and give a quantitative error estimate for the truncated ring polymer approximation. Then we propose the MLMC-PIMD method for calculating the truncated thermal average, which balances the total variance and the computational cost by optimizing the sampling numbers allocated to each sub-estimator. Extensive numerical tests are provided to show the approximation property of the truncated thermal average and the better performance of MLMC-PIMD method.

Further study can focus on the cases where the off-diagonal components of the potential function change sign or take complex values. In addition, we plan to test and further improve the proposed algorithm for realistic chemical applications with multi-dimensional potential surfaces.

\section*{Acknowledgements}
\hspace{4mm} This work has been partially supported by Beijing Academy of Artificial Intelligence (BAAI). Zhennan Zhou is supported by NSFC grant No. 11801016. The authors thank Prof. Jianfeng Lu and Jian-Guo Liu for helpful discussions.

\bibliographystyle{plain}
\bibliography{M335}

\appendix
\renewcommand{\appendixname}{Appendix~\Alph{section}}
\section{Calculating the optimal choice for $N_{k}$}\label{App-A}
\hspace{4mm} We define the Lagrangian function $$\mathcal{L}(N_{0},\cdots ,N_{k_{0}},\lambda)=\sum\limits_{k=0}^{k_{0}}2\tbinom{N}{2k}N_{k}-\lambda(\sum\limits_{k=0}^{k_{0}} \frac{4 C^{2 N} e^{2\left|C_{3}\right|}}{N_{k}((2 k) !)^{2}}-\epsilon),$$ and solve the equation
\begin{equation*}
\frac{\partial \mathcal{L}}{\partial N_{k}}=0\,(0\le k\le k_{0})\text{ and }\frac{\partial \mathcal{L}}{\partial \lambda}=0.
\end{equation*}
Notice
\begin{equation}\label{A-1}
\frac{\partial\mathcal{L}}{\partial N_{k}}=2\tbinom{N}{2k}+\lambda \frac{8 C^{2 N} e^{2\left|C_{3}\right|}}{N_{k}^{2}((2 k) !)^{2}}=0,
\end{equation}
and
\begin{equation}\label{A-2}
\frac{\partial \mathcal{L}}{\partial \lambda}=\epsilon-\sum\limits_{k=0}^{k=k_{0}}\frac{4 C^{2 N} e^{2\left|C_{3}\right|}}{N_{k}((2 k) !)^{2}}=0.
\end{equation}
By \eqref{A-1} we have
\begin{equation}\label{A-3}
N_{k}=\sqrt{-\lambda}\left(\tbinom{N}{2k}\right)^{-\frac{1}{2}}\frac{2C^{N}e^{|C_{3}|}}{(2k)!},
\end{equation}
and pass \eqref{A-3} to \eqref{A-2} we get
\begin{equation}\label{A-4}
\sum\limits_{k=0}^{k=k_{0}}\frac{2C^{N}e^{|C_{3}|}}{(2k)!}\sqrt{\tbinom{N}{2k}}(-\lambda)^{-\frac{1}{2}}=\epsilon\quad\Rightarrow\quad \lambda=-\frac{1}{\epsilon^{2}}\left(\sum\limits_{k=0}^{k=k_{0}}\frac{2C^{N}e^{|C_{3}|}}{(2k)!}\sqrt{\tbinom{N}{2k}}\right)^{2}.
\end{equation}
Using \eqref{A-3} and \eqref{A-4} we get
\begin{align}\label{A-5}
N_{k}
&=\frac{1}{\epsilon}\left(\sum\limits_{k=0}^{k=k_{0}}\frac{2C^{N}e^{|C_{3}|}}{(2k)!}\sqrt{\tbinom{N}{2k}}\right)\left(\tbinom{N}{2k}\right)^{-\frac{1}{2}}\frac{2C^{N}e^{|C_{3}|}}{(2k)!}\notag\\
&=\frac{4C^{2N}e^{2|C_{3}|}}{\epsilon}\left(\sum\limits_{k=0}^{k=k_{0}}\frac{1}{(2k)!}\sqrt{\tbinom{N}{2k}}\right)\left(\tbinom{N}{2k}\right)^{-\frac{1}{2}}\frac{1}{(2k)!},
\end{align}
where the total computational cost gets its minima. And from \eqref{A-5} we can see the optimal choice of $N_{k}$ is proportional to $\left(\tbinom{N}{2k}\right)^{-\frac{1}{2}}\frac{1}{(2k)!}$.

\end{document}